\documentclass[reqno]{amsart}

\newcommand{\la}{\langle} \newcommand{\ra}{\rangle}

\newcommand{\range}{\operatorname{R}}

\newcommand{\ls}{\lesssim} \newcommand{\gs}{\gtrsim}

\newcommand{\R}{\mathbb{R}} \newcommand{\C}{\mathbb{C}}
\newcommand{\N}{\mathbb{N}} 
\DeclareMathOperator{\supp}{supp} \DeclareMathOperator{\graph}{graph}

\usepackage{amsmath,amssymb,amsthm}

\makeatletter \def\@strippedMR{} \def\@scanforMR#1#2#3\endscan{%
  \ifx#1M\ifx#2R\def\@strippedMR{#3}%
  \else\def\@strippedMR{#1#2#3}%
  \fi\fi} \renewcommand\MR[1]{\relax \ifhmode\unskip\spacefactor3000
  \space\fi \@scanforMR#1\endscan
  MR\MRhref{\@strippedMR}{\@strippedMR}} \makeatother

\newtheorem{theorem}{Theorem} \newtheorem{cor}[theorem]{Corollary}
\newtheorem{proposition}[theorem]{Proposition}

\newtheorem{corollary}[theorem]{Corollary}
\newtheorem{assumption}[theorem]{Assumption}

\theoremstyle{remark} \newtheorem{remark}{Remark}

\theoremstyle{definition} 

 \DeclareMathOperator{\diam}{diam}

\numberwithin{equation}{section} \numberwithin{theorem}{section}

\begin{document}

\title[Convolutions of singular measures and the Zakharov
System]{Convolutions of singular measures and applications to the
  Zakharov system}

\author[I. Bejenaru]{Ioan Bejenaru} \address[I. Bejenaru]{Department
  of Mathematics, University of Chicago, Chicago, IL 60637, USA}
\email{bejenaru@math.uchicago.edu}

\author[S. Herr]{Sebastian Herr} \address[S. Herr]{Mathematisches
  Institut, Rheinische Friedrich-Wilhelms-Universit\"at Bonn,
  Endenicher Allee 60, 53115 Bonn, Germany \and Mathematisches
  Institut, Heinrich-Heine-Universit\"at D\"usseldorf,
  Universit\"atsstr. 1, 40225 D\"usseldorf, Germany}
\email{herr@math.uni-bonn.de}

\begin{abstract}
  Uniform $L^2$-estimates for the convolution of singular measures
  with respect to transversal submanifolds are proved in arbitrary
  space dimension. The results of Bennett-Bez are used to extend
   previous work of Bejenaru-Herr-Tataru. As an
  application, it is shown that the 3D Zakharov system is locally
  well-posed in the full subcritical regime.
\end{abstract}


\maketitle

\section{Introduction and main results}\label{sect:intro}
\noindent

In this paper we complete the development of a geometric
multilinear $L^2$-estimate which streamlines the analysis of a general
class of bilinear forms which appear in various types of nonlinear
PDE. In \cite{bejenaru_convolution_2010} Tataru and the authors proved uniform estimates
for the convolution of $L^2$ measures supported on transversal surfaces in three dimensions. 
This complemented the result of Bennett-Carbery-Wright in \cite{bennett_nonlinear_2005}.
In the present paper we generalize our previous result to higher dimensions
by using the recent work of Bennett-Bez in \cite{bennett_bez_2010}.

As an application, we establish a sharp result for the Zakharov
system in 3D.  Our result, when combined with the results in
\cite{bejenaru_zakharov_2009,ginibre_cauchy_1997}, closes the full
subcritical regime (in the sense of \cite[p. 387]{ginibre_cauchy_1997})
for the Zakharov system in all dimensions. As a consequence, the
remaining part of the paper is organized in two sections, each
containing results of independent interest.

\subsection{Convolutions of singular measures}
The first part of the paper is dedicated to a generalization to higher
dimensions of the results in \cite{bejenaru_convolution_2010}.  We
consider three subsets $\Sigma_1,\Sigma_2,\Sigma_3$ of submanifolds of
$\R^n$ whose codimensions add up to $n$ and which are transversal in
the sense that the normal spaces at each point span $\R^n$ and which
satisfy certain regularity assumptions. In this set-up we study the
restriction to $\Sigma_3$ of the convolution of two measures supported
on $\Sigma_1,\Sigma_2$. Our main results are global $L^2$ estimates.

We build on the result on nonlinear Brascamp-Lieb inequalities proved
in \cite{bennett_bez_2010}, see also
\cite{bennett_nonlinear_2005}. More precisely, we utilize the $m=3$
case of \cite[Theorem 1.3]{bennett_bez_2010} in order to extend the
trilinear case of \cite[Theorem 7.1]{bennett_bez_2010} to submanifolds
of general codimensions, formulated under global, quantitative
assumptions in the spirit of \cite{bejenaru_convolution_2010}.

Before we formulate the precise assumptions on the submanifolds, let
us introduce some notation: For numbers $m_1,m_2,m_3\in \N$ we define
the sets of indices $M_1=\{1,\ldots,m_1\}$,
$M_2=\{m_1+1,\ldots,m_1+m_2\}$, and
$M_3=\{m_1+m_2+1,\ldots,m_1+m_2+m_3\}$. Moreover, for a function
$\phi:\R^{n-m}\supset U \to \R^m$ we write
$\graph(\phi)=\{(x,\phi(x))^t\in \R^n: x\in U\}$.
\begin{assumption}\label{assumpt_main}
  There exist $0<\beta\leq 1$, $b>0$, $\theta>0$, $R>0$, and $m_i\in
  \N$ for $i=1,2,3$, with $n=m_1+m_2+m_3$, such that
  \begin{enumerate}
  \item\label{it:reg_cond} for every $i=1,2,3$ there exists an open
    set $U_i\subset \R^{n_i}$, $n_i=n-m_i$, and $\phi_i \in
    C^{1,\beta}(U_i;\R^{m_i})$ with the property
    \begin{equation}\label{eq:hoeldercond}
      R^{-\beta}\sup_{x\in U_i }|D\phi_i(x)|+ \sup_{x,x'\in U_i } \frac{|D\phi_i(x)-D\phi_i(x')|}{|x-x'|^{\beta}}
      \leq b,
    \end{equation}
    such that $\Sigma_i$ is relatively open, and compactly contained
    in $G_i \graph(\phi_i)$ for some orthogonal transformation $G_i\in
    O(n)$;

  \item\label{it:trans} for every $i=1,2,3$ and $\sigma_i \in\Sigma_i$
    and any orthonormal basis $\{\mathfrak{n}_{k}\}_{k \in M_i}$ of
    the normal space $N_{\sigma_i}(\Sigma_i)$ the determinant
    \[d(\sigma_1,\sigma_2,\sigma_3)=\det
    (\mathfrak{n}_{1}(\sigma_1),\ldots,\mathfrak{n}_{m_1}(\sigma_1),\ldots,\mathfrak{n}_{m_1+m_2+1}(\sigma_3),\ldots,\mathfrak{n}_{n}(\sigma_3))\]
    satisfies the uniform transversality condition
    \begin{equation}
      \label{eq:trans}
      \inf_{\sigma_1,\sigma_2,\sigma_3} |d(\sigma_1,\sigma_2,\sigma_3)| = \theta;
    \end{equation}
  \item\label{it:size} for every $i=1,2,3$ it holds
    \begin{equation}\label{eq:size}
      \diam(\Sigma_i)\leq R.
    \end{equation}
  \end{enumerate}
\end{assumption}

\begin{remark}
  The quantity $d(\sigma_1,\sigma_2,\sigma_3)$ is invariant under
  changes of the orthonormal bases within each normal space.
\end{remark}

We identify $f\in L^{2}(\Sigma_i)=L^{2}(\Sigma_i,\mu_i)$ --- $\mu_i$
being the $n_i$-dimensional Hausdorff-measure --- with the
distribution
\begin{equation*}
  \la f, \psi\ra=\int_{\Sigma_i} f(y) \psi(y) d\mu_i(y) ,\; \psi\in C_0^\infty(\R^n).
\end{equation*}
For $f\in L^{2}(\Sigma_1),g\in L^{2}(\Sigma_2)$ with compact support
the convolution $f\ast g$ is defined as the distribution
\begin{equation*}
  \la f\ast g,\psi \ra=\int_{\Sigma_1}\int_{\Sigma_2}f(x)g(y)\psi(x+y)d\mu_1(x)d\mu_2(y)
  , \; \psi \in C_0^\infty(\R^n).
\end{equation*}

Since a-priori the restriction of $f\ast g$ to sets of measure zero is
not well-defined, we begin with $f \in C_0(\Sigma_1)$ and $g \in
C_0(\Sigma_2)$. Then $f \ast g \in C_0(\R^n)$ and has a well-defined
trace on $\Sigma_3$. Once we have proved an appropriate $L^2$-bound,
the trace of $f \ast g $ on $\Sigma_3$ can be defined by density for
arbitrary $f \in L^{2}(\Sigma_1)$ and $g \in L^{2}(\Sigma_2)$.

Following the ideas of \cite{bejenaru_convolution_2010} we first note
the behavior under linear transformations.

\begin{proposition}\label{prop:conv_gen}
  Let $\Sigma_1$, $\Sigma_2$, $\Sigma_3$ satisfy the Assumption
  \ref{assumpt_main} with
  \[\theta\leq |d(\sigma_1,\sigma_2,\sigma_3)|\leq 2\theta,\]
  and suppose that the estimate
  \begin{equation}\label{eq:conv_gen_a}
    \|f\ast g\|_{L^2(\Sigma_3)}\leq C {\theta}^{-\frac12}
    \|f\|_{L^2(\Sigma_1)}\|g\|_{L^2(\Sigma_2)},
  \end{equation}
  holds true for all functions $f\in L^2(\Sigma_1)$, $g\in
  L^2(\Sigma_2)$.  If $T:\R^n\to \R^n$ is an invertible, linear map
  and $\Sigma_i'=T\Sigma_i$, then the estimate
  \begin{equation}\label{eq:conv_gen}
    \|f'\ast g'\|_{L^2(\Sigma_3')}\leq 2C {\theta'}^{-\frac12}
    \|f'\|_{L^2(\Sigma_1')}\|g'\|_{L^2(\Sigma_2')}
  \end{equation}
  holds true for all functions $f'\in L^2(\Sigma_1')$, $g'\in
  L^2(\Sigma_2')$, where
  \begin{equation*}
    \theta'=\inf_{\sigma_1',\sigma_2',\sigma_3'}|d'(\sigma_1',\sigma_2',\sigma_3')|
  \end{equation*}
  is defined in analogy to Assumption \ref{assumpt_main}
  \ref{it:trans}.
\end{proposition}
In summary, the size of the constant is determined only by the
transversality properties of the submanifolds.

Next, we look at the fully transversal case.  The dual formulation of
a local version of the following result for codimension $1$
submanifolds is contained in \cite[Theorem 7.1]{bennett_bez_2010}.

\begin{theorem}\label{thm:conv}
  Let $\Sigma_1,\Sigma_2,\Sigma_3$ be submanifolds in $\R^n$ which
  satisfy Assumption \ref{assumpt_main} with parameters $0<\beta\leq
  1$, $b=1$ and $\theta = \frac12$, and $R=1$. Then for each $f\in
  L^2(\Sigma_1)$ and $g\in L^2(\Sigma_2)$ the restriction of the
  convolution $f\ast g$ to $\Sigma_3$ is a well-defined
  $L^2(\Sigma_3)$-function which satisfies
  \begin{equation}\label{eq:conv}
    \|f\ast g\|_{L^2(\Sigma_3)}\leq C  \|f\|_{L^2(\Sigma_1)}\|g\|_{L^2(\Sigma_2)},
  \end{equation}
  where the constant $C$ depends only on $\beta$ and $n$.
\end{theorem}

Finally, in view of future applications let us note how the estimate
depends on the more general hypothesis of Assumption
\ref{assumpt_main}

\begin{cor}\label{cor:conv_theta}
  Let $\Sigma_1,\Sigma_2,\Sigma_3$ be submanifolds in $\R^n$ which
  satisfy Assumption \ref{assumpt_main} with parameters $0<\beta \leq
  1$, $b>0$, $0<\theta\leq 1/2$.  Then for each $f\in L^2(\Sigma_1)$
  and $g\in L^2(\Sigma_2)$ the restriction of the convolution $f\ast
  g$ to $\Sigma_3$ is a well-defined $L^2(\Sigma_3)$-function which
  satisfies
  \begin{equation}\label{eq:conv_theta}
    \|f\ast g\|_{L^2(\Sigma_3)}\leq C \theta^{-\frac12}
    \|f\|_{L^2(\Sigma_1)}\|g\|_{L^2(\Sigma_2)},
  \end{equation}
  where $C$ depends only on $\beta$, $n$, and the size of the quantity
  $R^\beta b \theta^{-1}$.
\end{cor}

\subsection{The 3D Zakharov system}
In this section we consider the initial value problem associated with
the Zakharov system
\begin{equation}\label{eq:zak}
  \begin{split}
    i\partial_t u+\Delta u= &n u \text{ in } (0,T)\times \R^3,\\
    \partial_t^2 n -\Delta n=&\Delta |u|^2 \text{ in } (0,T)\times \R^3,\\
    (u,n,\partial_t n)|_{t=0}\in & H^s(\R^3)\times
    H^\sigma(\R^3)\times H^{\sigma-1}(\R^3).
  \end{split}
\end{equation}

The Zakharov system is a model for Langmuir oscillations in a plasma,
cf. \cite{zakharov_collapse_1972} and \cite[Chapter 13]{Sulem-book} for more information.

Local weak solutions for \eqref{eq:zak} with smooth data were constructed by Sulem-Sulem in \cite{MR552204}, 
and local well-posedness for data in $H^2 \times H^1 \times L^2$ was established by Ozawa-Tsutsumi in \cite{ozawa_existence_1992}.
Provided that the Schr\"odinger part is small in $H^1$, global well-posedness for data in the energy space, 
see \cite{bourgain_wellposedness_1996} for details, 
was established by Bourgain-Colliander in \cite{bourgain_wellposedness_1996}.

We are interested in the low regularity well-posedness theory of
\eqref{eq:zak}. Our notion of well-posedness includes existence of
generalized solutions, uniqueness in a suitable subspace, local Lipschitz
continuity and persistence of initial regularity.  It has been shown by
Ginibre-Tsutsumi-Velo in \cite{ginibre_cauchy_1997} that
\eqref{eq:zak} is locally well-posed for $\sigma \geq 0$, $2s\geq
\sigma+1$, $\sigma \leq s \leq \sigma+1$. We extend this result to the full subcritical range in the
sense of \cite[p. 387]{ginibre_cauchy_1997}.

\begin{theorem}\label{thm:main-zak}
  The Cauchy problem \eqref{eq:zak} is locally well-posed in
  $H^s(\R^3)\times H^\sigma(\R^3)\times H^{\sigma-1}(\R^3)$ for
  $\sigma > -\frac12$, $\sigma \leq s \leq \sigma+1$, $2s>\sigma
  +\frac12$.
\end{theorem}

For a more detailed statement we refer the reader to \cite [Theorem
1.1]{bejenaru_zakharov_2009}.

The almost admissible endpoint $(s,\sigma)=(0,-\frac12)$, i.e. bottom
left corner of the convex region of admissible $(s,\sigma)$, matches
the 2D result obtained in \cite[Theorem 1.1]{bejenaru_zakharov_2009}
and extends the result of \cite[formula (1.10)]{ginibre_cauchy_1997}
for dimensions $d\geq 4$ to $d=3$. 

\section{Convolution estimates}

\begin{proof}[Proof of Proposition \ref{prop:conv_gen}]
  By density and duality, the claimed estimate is equivalent to
  \begin{equation}\label{gend}
    \begin{split}
      I(f,g,h):=& \int f(\sigma_1') g(\sigma_2') h(\sigma_3')
      \delta(\sigma_1'+\sigma_2'-\sigma_3')
      d\mu_1'(\sigma_1') d\mu_2'(\sigma_2')d\mu_3'(\sigma_3') \\
      \leq& 2C {\theta'}^{-\frac12} \|f\|_{L^2(\Sigma_1')}
      \|g\|_{L^2(\Sigma_2')} \|h\|_{L^2(\Sigma_3')},
    \end{split}
  \end{equation}
  for all nonnegative, continuous $f,g,h$. We assume that $\varphi_i :
  \R^{n_i}\supset \Omega_i \rightarrow \R^n$ is a global
  parametrization for $\Sigma_i, i=1,2,3$.  In that case
  $\varphi_i':=T \varphi_i$ is a parametrization for $\Sigma_i',
  i=1,2,3$.

  Using the above parameterizations \eqref{gend} is given as follows
  \begin{align*}
    I(f,g,h) =\int f(\varphi_1'(x)) g(\varphi_2'(y))
    h(\varphi_3'(z))(g_1'g_2'g_3')^{\frac12} d\nu'(x,y,z)
  \end{align*}
  where $g_i'=\det(D \varphi_i')^tD \varphi_i'$, and with respect to
  the measure
  \[
  d\nu'(x,y,z)=\delta(\varphi_1'(x)+\varphi_2'(y)-\varphi_3'(z)) dx dy
  dz.
  \]
  With $g_i=\det(D \varphi_i)^tD \varphi_i$ and the measure
  \[
  d\nu(x,y,z)=\delta(\varphi_1(x)+\varphi_2(y)-\varphi_3(z))
  (g_1g_2g_3)^{\frac12} dx dy dz
  \]
  an upper bound on $I(f,g,h)$ is given by
  \begin{align*}
    & \frac{\sup M(x,y,z)}{|\det T|} \int  \tilde{f}(\varphi_1(x)) \tilde{g}(\varphi_2(y)) \tilde{h}(\varphi_3(z))(g_1g_2g_3)^{\frac12} d\nu(x,y,z) \\
    \leq & 2\theta^{\frac12}\theta'^{-\frac12} \int  \tilde{f}(\sigma_1) \tilde{g}(\sigma_2) \tilde{h}(\sigma_3) \delta(\sigma_1+\sigma_2-\sigma_3) d\mu_1(\sigma_1) d\mu_2(\sigma_2) d\mu_3(\sigma_3) \\
    \leq &2C \theta'^{-\frac12} \|f\|_{L^2(\Sigma_1')}
    \|g\|_{L^2(\Sigma_2')} \|h\|_{L^2(\Sigma_3')},
  \end{align*}
  where we have used the definitions
  \[
  M = \prod_{i=1}^3 {g_i'}^{\frac14} g_i^{-\frac14}, \quad \tilde{f}=
  {g_1'}^{\frac14}g_1^{-\frac14} f(T \cdot),
  \]
  similarly for $\tilde{g}, \tilde{h}$.  We have also used that
  Dirac's $\delta$ obeys the simple rule
  \[
  \delta(T\varphi_1(x)+T\varphi_2(y)-T\varphi_3(z))= (\det T)^{-1}
  \delta(\varphi_1(x)+\varphi_2(y)-\varphi_3(z))
  \]
  and the following identity
  \begin{equation}\label{mn}
    \frac{M(x,y,z)}{\det T}
    = \left(\frac{d(\varphi_1(x),\varphi_2(y),\varphi_3(z))}{d'(\varphi_1'(x),\varphi_2'(y),\varphi_3'(z))}\right)^{\frac12},
  \end{equation}
  so that \eqref{mn} is the only claim which remains to be proved.

  For brevity, let $\sigma_1=\varphi_1(x)$, $\sigma_2=\varphi_2(y)$,
  $\sigma_3=\varphi_3(z)$, $\sigma_i'=T\sigma_i$, be arbitrary points
  on $\Sigma_i$, which will be fixed for the subsequent calculation.

  For $i=1,2,3$ we fix orthonormal bases
  $\{\mathfrak{n}_k(\sigma_i)\}_{k \in M_i}$ of the normal spaces and
  define the invertible matrix
  \[S=S(\sigma_1,\sigma_2,\sigma_3)=
  (\mathfrak{n}_1(\sigma_1),\ldots,\mathfrak{n}_{m_1}(\sigma_1),\ldots,
  \mathfrak{n}_{n_3+1}(\sigma_3),\ldots,\mathfrak{n}_{n}(\sigma_3))^t
  \]
  as well as $R=R(\sigma_1,\sigma_2,\sigma_3)=TS^{-1}$. Then, $T=RS$
  and $S$ has the property that if $\Sigma_i''=S \Sigma_i$ then $
  \{\mathfrak{e}_k\}_{k \in M_i} $ is an orthonormal basis of the
  normal space of $\Sigma_i''$ at $S\sigma_i$, $i=1,2,3$. We observe
  that
  \begin{equation}\label{eq:sr}
    \frac{\det((T D\varphi_i)^t T D \varphi_i)}{\det((D\varphi_i)^t D
      \varphi_i)}= \frac{\det((S D\varphi_i)^t S D
      \varphi_i)}{\det((D\varphi_i)^t D \varphi_i)} \, \cdot \, \frac{\det((RS
      D\varphi_i)^t RS D \varphi_i)}{\det((S D\varphi_i)^t S D \varphi_i)}.
  \end{equation}
  Thus, without restricting the generality of the problem, we can
  assume that an orthonormal basis of the normal space of $\Sigma_i'$
  at $\sigma_i'=T \sigma_i$ is given as $\{\mathfrak{e}_k\}_{k \in
    M_i}$, since this takes care of the first factor and it also
  provides the computation for the reverse situation which takes care
  of the second factor.

  Under this assumption the rows $n_k^t$ of $T$, i.e. $n_k:=T^t
  \mathfrak{e}_k$, $k \in M_i$ form a basis of the normal space of
  $\Sigma_i$ at $\sigma_i$, but not necessarily an orthonormal
  basis. We rely on two basic geometric facts: The first is that
  $(\det(A^t A))^{\frac12}$ is the $p$-dimensional volume of the
  parallelepiped spanned by the columns of $A\in \R^{n \times p}$.
  The second is that if
  \[
  A=(A_1|A_2), \, A_k \in \R^{n \times p_k},\, p_1+p_2=n, \text{ and }
  \range(A_1)\perp \range(A_2),
  \]
  then the volume of the parallelepiped spanned by the columns of $A$
  is the product of the volumes of the parallelepipeds spanned by the
  columns of $A_1$, $A_2$, respectively, i.e.
  \[
  \det(A) = (\det(A_1^t A_1))^{\frac12}(\det(A_2^t A_2))^{\frac12}.
  \]
  We define the submatrices
  \[N_i'=(\mathfrak{e}_{k_i+1},\ldots,\mathfrak{e}_{k_i+m_i}), \text{
    with } k_i \text{ such that } M_i=\{k_i+1,\ldots,k_i+m_i\},\] and
  $N_i =T^t N_i'$, where the columns $n_k$ are normal to $\Sigma_i$,
  but do not necessarily form an orthonormal set. We compute for
  $i=1,2,3$ based on the considerations above that
  \begin{align*}
    & \frac{\det((TD\varphi_i)^t T D\varphi_i)}{\det((D\varphi_i)^t
      D\varphi_i)}
    = \frac{\det(T D\varphi_i)^t(T D\varphi_i) \det(N_i^t N_i)}{\det^2(D\varphi_i| N_i)} \\
    =& \frac{\det^2(T) \det(T D\varphi_i)^t(T D\varphi_i) \det(N_i^t
      N_i)}{\det^2(TD\varphi_i|TN_i)}
  \end{align*}
  where here in in the sequel we suppress the evaluation of
  $\varphi_i$ at $x,y,z$, respectively. Next, we use
  \[
  \det(TD\varphi_i|TN_i)=\det(TD\varphi_i|P_i TN_i),
  \]
  where $P_i$ is the orthogonal projection onto $N_i'$, and conclude
  \[\det(TD\varphi_i|TN_i)=(\det((P_i TN_i)^t P_i
  TN_i))^{\frac12}(\det((TD\varphi_i)^t TD\varphi_i))^{\frac12},
  \]
  such that in summary
  \begin{align*}
    \frac{\det((TD\varphi_i)^t T D\varphi_i)}{\det((D\varphi_i)^t
      D\varphi_i)} = \frac{\det^2(T)\det(N_i^t N_i)}{\det((P_i
      TN_i)^tP_i TN_i)} = \frac{\det^2 T}{\det(N_i^t N_i)}.
  \end{align*}
  In the last step we have used the particular form of the vectors in
  $N_i'$ and the fact that $T^t=(N_1|N_2|N_3)$, which yields
  \[
  \det((P_i TN_i)^tP_i TN_i)={\det}^2(N_i^t N_i).
  \]
  The above computation holds for all $i \in \{1,2,3\}$, therefore
  \[\frac{M(x,y,z)}{\det T}=(\det T)^{-1} \prod_{i=1}^3
  \left(\frac{\det^2 T}{\det(N_i^t N_i)}\right)^{\frac14} =
  \left(\frac{\det(N_1|N_2|N_3)}{\prod_{i=1}^3 (\det(N_i^t
      N_i))^{\frac12}}\right)^{\frac12}
  \]
  This expression is invariant with respect to the choice of normal
  vectors in $N_i$, hence we can use an orthonormal set to obtain
  \[\frac{M(x,y,z)}{\det T}=\left(
    d(\varphi_1(x),\varphi_2(y),\varphi_3(z))\right)^{\frac12}
  \]
  and in view of our previous reduction in \eqref{eq:sr} the claim
  \eqref{mn} follows. This ends the proof of Proposition
  \ref{prop:conv_gen}.
\end{proof}

\begin{proof}[Proof of Theorem \ref{thm:conv}]
  In what follows we use Landau's notation $o(1)$ for scalars, vectors
  or matrices to denote a quantity which can be made arbitrarily small
  as $R=\max(\diam(\Sigma_1),\diam(\Sigma_2),\diam(\Sigma_3))\to 0$.
  For brevity we introduce the shorthand notation
  \[
  (x_{i},\ldots, x_j)^t= x_{i,j}, \quad i<j.
  \]
  We subdivide the proof into two steps:

  {\em Step 1.} By a finite partition (depending only on the
  dimension), linear changes of coordinates as in the proof of
  Corollary \ref{cor:conv_theta} below we can reduce the problem to
  the following set-up: There exists a triplet
  $(\sigma_1^0,\sigma_2^0,\sigma_3^0) \in \Sigma_1 \times \Sigma_2
  \times \Sigma_3$ where $\{\mathfrak{e}_k\}_{k \in M_i}$ is a basis
  for $N_{\sigma^0_i}(\Sigma_i)$, $i=1,2,3$, such that by the implicit
  function theorem we have $C^{1,\beta}$-parametrizations
  $\varphi_i:\Omega_i \to \R^n$, for open subsets $\Omega_i$ of the
  unit ball in $\R^{n_i}$, centered at $a_i^0$, given as
  \begin{align*}
    \varphi_1(x_{m_1+1,n}) =&(\phi_1(x_{m_1+1,n}), \ldots, \phi_{m_1}(x_{m_1+1,n}), x_{m_1+1,n})^t, \\
    \varphi_2(x_{1,m_1}, x_{m_1+m_2+1,n})
    =&(x_{1,m_1}, \phi_{m_1+1}(x_{1,m_1}, x_{m_1+m_2+1,n}),\ldots,\\& \phi_{m_1+m_2}(x_{1,m_1}, x_{m_1+m_2+1,n}), x_{m_1+m_2+1,n})^t,\\
    \varphi_3(x_{1,n_3})=&(x_{1,n_3}, \phi_{n_3+1}(x_{1,n_3}),
    \ldots,\phi_n(x_{1,n_3}))^t,
  \end{align*}
  where $m_i=n-n_i$, such that $\Sigma_i=\varphi_i(\Omega_i)$,
  $\diam(\Sigma_i)$ is small enough, $\varphi_i(a_i^0)=\sigma_i^0$,
  where the submanifolds intersect
  $\varphi_1(a_1^0)+\varphi_2(a_2^0)=\varphi_3(a_3^0)$, and
  \begin{equation} \label{inder}
    \partial_{l} \phi_j(a_i^0)= 0, \text{ for all } j \in M_i, \text{
      and all } 1\leq l\leq n_i.
  \end{equation}

  {\em Step 2.} We have that for each $i \in \{1,2,3\}$
  \begin{equation} \label{DP} \det[D\varphi_i^t D\varphi_i ](a^0_i) =
    1, \quad \det[D\varphi_i^t D\varphi_i ] =1+o(1)
  \end{equation}
  and the determinant of the normals satisfies
  \[
  d(\sigma_1^0,\sigma_2^0,\sigma_3^0)=1, \quad
  d(\sigma_1,\sigma_2,\sigma_3) = 1+o(1).
  \]
  In this set-up, we need to estimate
  \begin{align*}
    & \int (f \circ \varphi_1)(x_{m_1+1,n}) (g \circ \varphi_2) (x_{1,m_1}, y_{m_1+m_2+1,n}) (h \circ \varphi_3) (y_{1,n_3}) \\
    & \delta(\varphi_1(x_{m_1+1,n}) + \varphi_2 (x_{1,m_1}, y_{m_1+m_2+1,n})-\varphi_3(y_{1,n_3}))\\
    & (\det[D\varphi_1^t D\varphi_1 ]\det[
    D\varphi_2^tD\varphi_2]\det[ D\varphi_3^t D\varphi_3])^{\frac12}
    dx_{1,n} d y_{1,n}
  \end{align*}
  For the function
  \[F(x_{1,n}, y_{1,n})= \varphi_1(x_{m_1+1,n}) + \varphi_2
  (x_{1,m_1}, y_{m_1+m_2+1,n})-\varphi_3(y_{1,n_3})\] it follows from
  the implicit function theorem that there exists a $C^{1,\beta}$
  function $G$ such that $F(x_{1,n}, y_{1,n})=0$ if and only if
  $y_{1,n}=G(x_{1,n})$, since
  \begin{equation} \label{detdelta} | \det \partial_{y_{1,n}} F| =
    1+o(1),
  \end{equation} because \eqref{inder} yields that the matrix is close to the diagonal
  matrix with $-1$ as the first $n_3$ diagonal entries and $+1$ as the remaining $m_3$ diagonal entries. Since the following is
  true
  \[
  \delta(F(x_{1,n}, y_{1,n})) = | \det\partial_{y_{1,n}} F|^{-1}
  \delta(y_{1,n}-G(x_{1,n})),
  \]
  the above integral the above integral can be rewritten as
  \begin{align*}
    &\int (f \circ \varphi_1) (x_{m_1+1,n}) (g \circ \varphi_2) ( x_{1,m_1}, G_{m_1+m_2+1,n}(x_{1,n}))\\
    &\qquad (h \circ \varphi_3)(G_{1,n_3}(x_{1,n})) m(x_{1,n}) d
    x_{1,n}
  \end{align*}
  where $m(x_{1,n})= 1+o(1)$ in the domain of integration, which
  follows from \eqref{DP} and \eqref{detdelta}.  Then, following the
  ideas in \cite{bennett_bez_2010,bennett_nonlinear_2005}, we define
  the maps $B_i : \R^n \rightarrow \R^{n_i}$ by
  \begin{align*}
    B_1 x_{1,n}&= x_{m_1+1,n}, \qquad B_2 x_{1,n}=(x_{1,m_1}, G_{m_1+m_2+1,n}(x_{1,n})), \\
    B_{3}x_{1,n}&= G_{1,n_3}(x_{1,n}).
  \end{align*}
  From the properties of $\varphi_i$ and \eqref{DP} it follows that
  $B_1,B_2,B_3$ are $C^{1,\beta}$ functions.  With these notations the
  above integral becomes
  \begin{align*}
    \int (f \circ \varphi_1) (B_1 x_{1,n}) (g \circ \varphi_2) ( B_2
    x_{1,n}) (h \circ \varphi_3)(B_3 x_{1,n}) m(x_{1,n})d x_{1,n}
  \end{align*}
  Next, we will verify the assumptions of \cite[Theorem
  1.3]{bennett_bez_2010} on the kernels of $D B_i(x_0)$, where
  $x_0=([a_2^0]_{1,m_1},a_1^0)\in \R^n$. We start with $i=1$:
  \[
  DB_1(x_0) =\begin{pmatrix} 0& I_{n_1}
  \end{pmatrix}
  \]
  hence an orthonormal basis of $\ker DB_1(x_0)$ is of the form
  $\{\mathfrak{e}_{k}\}_{k \in M_1}$. For $i=2$ we compute
  \[
  DG(x_0)
  =-[(\partial_{y_{1,n}}F(x_0,G(x_0)))^{-1}\partial_{x_{1,n}}F(x_0,G(x_0)]=
  \begin{pmatrix}
    I_{n_3}& 0\\
    0& -I_{m_3}
  \end{pmatrix}
  \]
  which implies
  \[
  DB_2(x_0)=\begin{pmatrix}
    I_{m_1} & 0\\
    0 & -I_{m_3}
  \end{pmatrix}
  \]
  and an orthonormal basis of $\ker DB_2(x_0)$ is of the form
  $\{\mathfrak{e}_k\}_{k \in M_2}$.  Concerning $i=3$, the computation
  of $DG(x_0)$ above immediately yields
  \[
  DB_3(x_0)=\begin{pmatrix} I_{n_3}& 0
  \end{pmatrix}
  \]
  and an orthonormal basis of $\ker DB_3(x_0)$ is given as
  $\{\mathfrak{e}_k\}_{k \in M_3}$.

  From the above characterizations of the kernels of $dB_j$, it
  follows from \cite[formula (25)]{bennett_bez_2010} that
  \[
  \left|\star \bigwedge_{j=1}^3 \star X_j(DB_j(x_0))\right| = 1,
  \]
  where we use the notation of \cite{bennett_bez_2010}. This allows us
  to invoke the result of \cite[Theorem 1.3]{bennett_bez_2010} in a
  small neighborhood of $x_0$, whose size depends only on $\beta$ and
  $n$.
\end{proof}

For the remaining proof we will follow closely the argument in
\cite[Proof of Corollary 1.6]{bejenaru_convolution_2010}.

\begin{proof}[Proof of Corollary \ref{cor:conv_theta}] {\em Step 1.}
  We first carry out the proof under the additional hypothesis
  \begin{equation}\label{eq:rbt_strong}
    R^\beta b \theta^{-1}\ll  1.
  \end{equation}
  We have $\|D\phi_i\|\leq b R^\beta\ll 1$ throughout $U_i\subset
  \R^{n_i}$.

  Let $i=1,2,3$ and $\sigma_i^0\in \Sigma_i$ be fixed. Define the
  normal vectors $\{n_k(\sigma_i)\}_{k \in M_i}$ at $\sigma_i
  =G_i\cdot (x,\phi_i(x))^t$ to be the columns of the matrix
  \[
  G_i \begin{pmatrix}-D\phi_i^t(x)\\I_{m_i}\end{pmatrix}\in
  \R^{n\times m_i}.
  \]
  These vectors satisfy
  \begin{equation}
    |n_k(\sigma_i) - n_k (\sigma_i^0)| \ls b
    R^{\beta} \ll \theta,
    \label{eq:normals_diff}\end{equation}
  for all $k \in M_i$, $i = 1,2,3$.  By the Gram-Schmidt
  orthonormalization procedure, we can also construct from
  $\{n_k(\sigma_i)\}_{k \in M_i}$ an orthonormal basis
  $\{\mathfrak{n}_k(\sigma_i)\}_{k \in M_i}$ of the normal space at
  $\sigma_i\in \Sigma_i$ satisfying \eqref{eq:normals_diff}, which
  shows that
  \begin{equation}
    |d(\sigma_1^0,\sigma_2^0,\sigma_3^0)- d(\sigma_1,\sigma_2,\sigma_3)|
    \ll \theta.
    \label{eq:det_diff}\end{equation}  
  Moreover, we observe that
  \[
  |(\sigma_i -\sigma_i^0)\cdot  n_k(\sigma_i^0) | \ls bR^{1+\beta}
  \ll R \theta, \quad k \in M_i,
  \]
  which shows that $\Sigma_i$ is contained in a plain layer of
  thickness $\ll R \theta$ with respect to the $n_k(\sigma_i^0)$
  direction, for all $k \in M_i$. By $L^2(\R^n)$-orthogonality with
  respect to such layers it suffices to prove the desired bound
  \eqref{eq:conv_theta} in the case when the other two submanifolds
  are contained in similar regions, i.e.
  \begin{equation}
    |(\sigma_i -\sigma_i^0) \cdot  n_k(\sigma_j^0) | \ll R \theta,
    \quad k\in M_j, \, i,j= 1,2,3.
    \label{eq:planeloc}
  \end{equation}
  We will apply Proposition \ref{prop:conv_gen} with the matrix
  \begin{equation*}
    T=  R\theta \, 
    (A^t)^{-1}, \quad A =(n_1(\sigma_1^0),\ldots,n_{m_1}(\sigma_1^0),\ldots,
    n_{n_3}(\sigma_3^0),\ldots,n_{n}(\sigma_3^0)).
  \end{equation*}
  It remains to show that the submanifolds
  $\tilde{\Sigma}_i:=T^{-1}\Sigma_i$ satisfy the assumptions of
  Theorem \ref{thm:conv}, i.e.
  \begin{enumerate}
  \item\label{it:s} the size condition $\diam(\tilde{\Sigma}_i)\leq
    1$,
  \item\label{it:t} the transversality condition \eqref{eq:trans} with
    $\theta = \frac12$,
  \item\label{it:h} the regularity condition \eqref{eq:hoeldercond}
    with $R=b=1$.
  \end{enumerate}
  Concerning item \ref{it:s} we observe that
  \[
  T^{-1} (\sigma_i -\sigma_i^0) = \frac{1}{R\theta} ( 
  n_1(\sigma_1^0)\cdot (\sigma_i -\sigma_i^0) , \ldots, 
  n_n(\sigma_3^0)\cdot (\sigma_i -\sigma_i^0) )^t,
  \]
  such that \eqref{eq:planeloc} shows $\diam(\tilde{\Sigma}_i)\leq 1$.

  In order to obtain the transversality condition in \ref{it:t}, we
  estimate
  \begin{equation}\label{eq:norm_inverse}
    \|A^{-1}\|\ls |\det A|^{-1}\sim \theta^{-1}, \quad \|T\|\ls R.
  \end{equation}
  Let $k \in M_i$. We define at $\tilde{\sigma}_i\in \tilde{\Sigma}_i$
  a normal vector $\tilde{n}_k(\tilde{\sigma}_i)$ to
  $\tilde{\Sigma}_i$ by
  \begin{equation}\label{eq:change_normals}
    \tilde{n}_k(\tilde{\sigma}_i)=
    A^{-1}n_k(T\tilde{\sigma}_i).
  \end{equation}
  By construction for $\tilde\sigma_i^0 = T^{-1} \sigma_i^0$ we have $
  \tilde{n}_k(\tilde{\sigma}_i^0) = \mathfrak{e}_k$. By
  \eqref{eq:normals_diff} and \eqref{eq:norm_inverse} it follows that
  \begin{equation}\label{eq:new_n}
    |\tilde{n}_k(\tilde{\sigma}_i) -\mathfrak{e}_k| \ls R^\beta b\ll \theta.
  \end{equation}
  Thus, we have found a basis $\{\tilde{n}_k(\tilde{\sigma}_i)\}_{k\in
    M_i}$ of $N_{\tilde{\sigma}_i}(\tilde{\Sigma}_i)$.  By the
  Gram-Schmidt process, we can recursively construct an orthonormal
  basis $\{\tilde{\mathfrak{n}}_k(\tilde{\sigma}_i)\}_{k\in M_i}$ with
  the property \eqref{eq:new_n}.  This in turn yields the desired
  transversality condition
  \[
  \tilde{d}(\tilde{\sigma}_1,\tilde{\sigma}_2,\tilde{\sigma}_3)\geq
  1/2.
  \]
  Concerning the regularity condition in \ref{it:h} we define
  \[\tilde{\Phi}_i(p)=(R\theta)^{-1}[(G^{-1}_i T
  p)_{n_i+1,n}-\phi_i((G^{-1}_i Tp)_{1,n_i})],\] such that with
  $Q_i=T^{-1}G_i (U_i\times \R^{m_i})$ it is
  \[\tilde{\Sigma}_i=\{p \in Q_i\subset \R^n:\tilde{\Phi}_i(p)=0\}.\]
  We would like to resolve this equation for $p_k$, $k \in M_i$. Now,
  for $k\leq l$ let $I_{k,l}$ be the $l-k+1\times n$ matrix, such that
  $I_{k,l}p=p_{k,l}$. It is
  \[
  D\tilde{\Phi}_i (p)=I_{n_i+1,n}G_i^t(A^t)^{-1}-D\phi_i((G_i^t
  Tp)_{1,n_i})I_{1,n_i}G_i^t(A^t)^{-1}.
  \]
  To keep the exposition clear we discuss the case $i=1$ only.
  \[
  D_{1,m_1}\tilde{\Phi}_1 (p)=\left(I_{n_1+1,n}G_1^t
    -D\phi_1((G_1^tTp)_{1,n_1})I_{1,n_1}G_1^t\right)(A^t)^{-1}I_{1,m_1}^t.
  \]
  Since $\|D\phi_1\|\ll 1$ in $U_1$ and by construction of $A$ it
  holds
  \[
  \|I_{n_1+1,n}G_1^t
  -D\phi_1((G_1^tTp)_{1,n_1})I_{1,n_1}G_1^t-I_{1,m_1}A^t\|\ll 1,
  \]
  which shows that
  \[
  \|D_{1,m_1}\tilde{\Phi}_1 (p)-I_{m_1}\|\ll 1.
  \]
  It also implies that for
  \[
  D_{m_1+1,n}\tilde{\Phi}_1 (p)=\left(I_{n_1+1,n}G_1^t
    -D\phi_1((G_1^tTp)_{1,n_1})I_{1,n_1}G_1^t\right)(A^t)^{-1}I_{m_1+1,n}^t
  \]
  we have
  \[
  \|D_{m_1+1,n}\tilde{\Phi}_1 (p)\|\ll 1.
  \]
  At $p=\tilde{\sigma}_1^0$ we evaluate
  \[D_{1,m_1}\tilde{\Phi}_1(\tilde{\sigma}_1^0)=I_{m_1} \text{ and
  }D_{m_1+1,n}\tilde{\Phi}_1(\tilde{\sigma}_1^0)=0.\] The implicit
  function theorem yields a global resolution $\tilde{\phi_1}\in
  C^{1,\beta}(\tilde{U}_1)$ with domain $\tilde{U}_1=I_{1,n_1}(Q_1)$
  such that $\tilde{\Phi}_1(\tilde{\phi}_1(\tilde{x}),\tilde{x})=0$
  with $D\tilde{\phi}_1(\tilde{x}^0)=0$ and the analog of
  \eqref{eq:hoeldercond} is satisfied with $R=b=1$.

  {\em Step 2.} Finally, we remove the additional assumption
  \eqref{eq:rbt_strong}. In general we have $R^\beta b \theta^{-1}\gs
  1$. We partition each submanifold $\Sigma_i$ into about
  $R\delta^{-1}$ pieces of diameter $\delta$ for $\delta^\beta b \ll
  \theta$. It remains to prove that for each such piece we can find a
  graph representation satisfying Assumption \ref{assumpt_main}
  \ref{it:reg_cond} with $R$ replaced with $\delta$. In order to to
  so, in each piece we select a point $G_i(a_i^0,\phi_i(a_i^0))^t$ and
  define a rotation $O_i\in \R^{n \times n}$ with the property
  \[
  O_i \, \mathrm{range}
  \begin{pmatrix} 0\\ I_{m_i}
  \end{pmatrix}= \mathrm{ range}
  \begin{pmatrix} -D\phi_i^t(a_i^0)\\I_{m_i}
  \end{pmatrix}
  ,
  \]
  and the implicit function theorem yields a representation of the
  piece as $G_iO_i \graph(\tilde{\phi}_i)$ with vanishing differential
  at a point. This implies \eqref{eq:hoeldercond} with $R$ replaced by
  $\delta$.
\end{proof}

\section{The Zakharov system}\label{sect:zakharov}
\subsection{Notation and function spaces}\label{subsect:not}
\noindent
We adopt the notation from \cite{bejenaru_zakharov_2009}: We write
$A\ls B$ if there exists a harmless constant $c>0$ such that $A\leq c
B$. Moreover, we write $A\gs B$ if $B\ls A$ and $A\sim B$ if $A\ls B$
and $A\gs B$. Throughout this paper we will denote dyadic numbers
$2^n$ for $n \in \N$ by the corresponding upper-case letters, e.g.
$N=2^n,L=2^l,\ldots$.

Let $\psi\in C^\infty_0((-2,2))$ be an even, non-negative function
with the property $\psi(r)=1$ for $|r|\leq 1$.  We use it to define a
partition of unity in $\R$,
\[
1 = \sum_{N \geq 1}\psi_N, \; \psi_1 = \psi, \;
\psi_N(r)=\psi\left(\frac{r}{N}\right)-\psi\left(\frac{2r}{N}\right),
\; N=2^n \geq 2.
\]
Thus $\supp\psi_1\subset [-2,2]$ and $\supp{\psi}_N\subset
[-2N,-N/2]\cup[N/2,2N]$ for $N\geq 2$. For $f:\R^3 \to \C$ we define
the dyadic frequency localization operators $P_N$ by
\begin{equation*}
  \mathcal{F}_x(P_N f)(\xi) =  \psi_N(|\xi|) \mathcal{F}_x f(\xi).
\end{equation*}
For $u:\R^3\times \R \to \C$ we define $(P_N u)(x,t)=(P_N
u(\cdot,t))(x)$. We will often write $u_N=P_Nu$ for brevity.  We
denote the space-time Fourier support of $P_N$ by the corresponding
Gothic letter
\begin{align*}
  \mathfrak{P}_1=&\left\{(\xi,\tau) \in \R^3\times \R \mid |\xi|\leq 2 \right\},\\
  \mathfrak{P}_N=&\left\{(\xi,\tau) \in \R^3\times \R \mid N/2\leq
    |\xi|\leq 2N \right\}.
\end{align*}

Moreover, for dyadic $L\geq 1$ we define the modulation localization
operators
\begin{align}
  \mathcal{F} (S_L u)(\tau,\xi) &= \psi_L(\tau+|\xi|^2)
  \mathcal{F} u(\tau,\xi) \qquad \text{(Schr\"odinger case)}, \label{eq:s_mod}\\
  \mathcal{F} (W^\pm_L u)(\tau,\xi) &= \psi_L(\tau\pm |\xi|)
  \mathcal{F} u(\tau,\xi) \qquad \text{(Wave case)},\label{eq:w_mod}
\end{align}
and the corresponding space-time Fourier supports
\begin{align*}
  \mathfrak{S}_1=&\left\{(\xi,\tau) \in \R^3\times \R \mid |\tau+|\xi|^2|\leq 2 \right\},\\
  \mathfrak{S}_L=&\left\{(\xi,\tau) \in \R^3\times \R \mid L/2\leq
    |\tau+|\xi|^2|\leq 2L \right\},
\end{align*}
respectively
\begin{align*}
  \mathfrak{W}^{\pm}_1=&\left\{(\xi,\tau) \in \R^3\times \R \mid |\tau \pm |\xi||\leq 2 \right\},\\
  \mathfrak{W}^{\pm}_L=&\left\{(\xi,\tau) \in \R^3\times \R \mid
    L/2\leq |\tau \pm |\xi||\leq 2L \right\}.
\end{align*}

Next we introduce the decompositions with respect to angular
variables. For each $A \in \N$ we choose a decomposition $\{
\omega_A^j \}_{j \in \Omega_A}$ of $\mathbb{S}^2$ with the following
properties:
\begin{enumerate}
\item Each $\omega_A^j$ is a spherical cap with angular opening
  $A^{-1}$, i.e the angle $\angle(x,y)$ between any two vectors in
  $x,y\in \omega_A^j$ satisfies \[|\angle(x,y)|\leq A^{-1}.\] \item
  $\mathbb{S}^2$ is the almost disjoint union of $\{ \omega_A^j \}_{j
    \in \Omega_A}$, i.e. if $\chi_{\omega_A^j}$ denotes the
  characteristic function of the cap $\omega_A^j$, then have the
  following
  \[
  1\leq \chi(x):=\sum_{j \in \Omega_A} \chi_{\omega_A^j}(x)\leq 3,
  \qquad \forall x \in \mathbb{S}^2,
  \]
  and we require that any two centers of caps in our collection are
  separated by a distance $\sim A^{-1}$ such that $\#\Omega_A\ls A^2$.
\end{enumerate}

Related to this we define the function
\[\alpha(j_1,j_2)=\inf\{|\angle(\pm x, y)|: x\in \omega_A^{j_1},y\in
\omega_A^{j_2}\}
\]
which measures the minimal angle between any two straight lines
through the caps $\omega_A^{j_1}$ and $\omega_A^{j_2}$, respectively.

Based on the above construction, for each $j \in \Omega_A$ we define
\begin{equation*}
  \mathfrak{Q}^{j}_A=\left\{(\xi,\tau) \in \R^3 \setminus \{ 0 \} \times \R :  \frac{\xi}{|\xi|} \in \omega_A^j \right\}.
\end{equation*}
and the corresponding localization operator
\[
\mathcal{F}(Q_A^j u)(\xi,\tau) =
\frac{\chi_{\omega_A^j}(\frac{\xi}{|\xi|} )}{\chi(\frac{\xi}{|\xi|} )}
\mathcal{F} u(\xi,\tau).
\]

For $k,\ell \in \R$ and $T>0$ we define the space
$\mathbf{Z}^{k,\ell}_T$ as the Banach space of all pairs of space-time
distributions $(u,n)$ which satisfy
\begin{equation}
  \label{eq:x-def}
  \begin{split}
    u &\in C( [0,T]; H^k(\mathbb{R}^3;\C)),\\
    n &\in C( [0,T]; H^\ell(\mathbb{R}^3;\R)) \cap C^1( [0,T];
    H^{\ell-1}(\mathbb{R}^3;\R)),
  \end{split}
\end{equation}
endowed with the standard norm $\|\cdot \|_{\mathbf{Z}^{k,\ell}_T}$
defined as
\begin{equation}\label{eq:x-norm}
  \|(u,n)\|^2_{\mathbf{Z}^{k,\ell}_T} = \sup_{t \in [0,T]}\Big\{
  \|u(t)\|^2_{H_x^k} + \|n(t)\|^2_{H_x^{\ell}} +  \|\partial_t n(t)\|^2_{H_x^{\ell-1}}\Big \}.
\end{equation}

Let $\sigma,b\in \R$, $1\leq p< \infty$. In connection to the operator
$i\partial_t + \Delta$ we define the Bourgain space $X^S_{\sigma,b,p}$
of all $u \in \mathcal{S}'(\R^3\times \R)$ for which the norm
\[
\|u\|_{X^S_{\sigma,b,p}}=\left(\sum_{N\geq
    1}N^{2\sigma}\left(\sum_{L\geq 1}L^{pb}
    \|S_LP_Nu\|_{L^2}^p\right)^{\frac2p}\right)^{\frac12}
\]
is finite. Similarly, to the half-wave operators $i\partial_t \pm
\langle \nabla\rangle$ we associate the Bourgain spaces
$X^{W^\pm}_{\sigma,b,p}$ of all $v \in \mathcal{S}'(\R^3\times \R)$
for which the norm
\[
\|v\|_{X^{W^\pm}_{\sigma,b,p}} = \left(\sum_{N\geq
    1}N^{2\sigma}\left(\sum_{L \geq 1}L^{pb}
    \|W^\pm_LP_Nu\|_{L^2}^p\right)^{\frac2p}\right)^{\frac12}
\]
is finite. For $p=\infty$ we modify the definition as usual.  In cases
where the Schwartz space $\mathcal{S}(\R^3\times \R)$ is not dense in
$X^{W^\pm}_{\sigma,b,p}$ or $X^{S}_{\sigma,b,p}$, respectively, we
redefine the spaces and take the closure of $\mathcal{S}(\R^3\times
\R)$ instead.

For a normed space $B\subset \mathcal{S}'(\R^n\times \R;\C)$ of
space-time distributions we denote by $\overline{B}$ the space of
complex conjugates with the induced norm.

For $T>0$ we define the space $B(T)$ of restrictions of distributions
in $B$ to the set $\R^n\times (0,T)$ with the induced norm
\[
\|u\|_{B(T)}=\inf \{ \| \tilde{u} \|_{B} : \; \tilde{u} \in B \text{
  is an extension of } u \text{ to }\R^n \times \R\}.
\]

\subsection{Multilinear estimates}\label{subsect:mult-est}
\noindent
This section is devoted to the proof of the crucial multilinear
estimates which imply the well-posedness result for the Zakharov
system in Theorem \ref{thm:main-zak}. The detailed reduction to multilinear
estimates as explained in \cite[Section 3]{bejenaru_zakharov_2009}
remains true verbatim, cf. also \cite{ginibre_cauchy_1997}, so we do
not reproduce it here. Given these multilinear estimates, Theorem
\ref{thm:main-zak} can be deduced by the standard Picard iteration
argument as described in \cite[Section 5]{bejenaru_zakharov_2009} for
the $2d$ case. Therefore, in the sequel we will focus on the proof of the following:

\begin{theorem}\label{thm:tri} Assume that $s > 0, \sigma > -\frac12,
  \sigma \leq s \leq \sigma+1$, $\sigma - 2s < -\frac12$.
  \begin{enumerate}\item\label{it:tri_a} For all $0<T\leq 1$ and for
    all functions $u,u_1,u_2 \in X^S_{s,\frac12,1}(T)$ and $v\in
    X^{W+}_{\sigma,\frac{1}{2},1}(T)$ the following estimates hold
    true:
    \begin{align}
      \| uv\|_{X^S_{s,-\frac12,1}(T)} \ls{} & \|u\|_{X^S_{s,\frac12,1}(T)}\|v\|_{X^{W+}_{\sigma,\frac12,1}(T)}\label{eq:sws1},\\
      \| u\bar{v}\|_{X^S_{s,-\frac12,1}(T)} \ls{} & \|u\|_{X^S_{s,\frac12,1}(T)} \|v\|_{X^{W+}_{\sigma,\frac12,1}(T)} \label{eq:sws2},\\
      \left\| \frac{\Delta}{\langle\nabla\rangle} ( u_1
        \bar{u}_2)\right\|_{X^{W+}_{\sigma,-\frac12,1}(T)} \ls{} &
      \|u_1\|_{X^S_{s,\frac12,1}(T)}
      \|u_2\|_{X^S_{s,\frac12,1}(T)}.\label{eq:wss}
    \end{align}
  \item\label{it:tri_b} There exists $\theta = \theta(s,\sigma) > 0$
    in the above regime for $s,\sigma$ such that all the inequalities
    can be improved with a factor of $T^\theta$ on the right hand
    side.
  \end{enumerate}
\end{theorem}

We have split the above result in two parts for the following reason.
Part \ref{it:tri_a} contains the "clean" estimates without keeping
track of the gains of powers of $T$ which may distract the reader from
the main ideas.  However, from part \ref{it:tri_a}, we would be able
to claim only a small data result for the Zakharov system.  It is part
\ref{it:tri_b} that allows us to claim the local well-posedness result
for large data.

We introduce the notation
\[
I(f,g_1,g_2) = \int f(\zeta_1-\zeta_2) g_1(\zeta_1) g_2(\zeta_2)
d\zeta_1 d\zeta_2,
\]
where $\zeta_i=(\xi_i,\tau_i)$, $i=1,2$.  Using duality and the fact
that $\overline{\mathcal{F}u}=\mathcal{F}\overline{u}(-\cdot)$, we can
reduce Theorem \ref{thm:tri} to the following trilinear estimates:
\begin{proposition}\label{prop:trilinear} Assume that $s > 0, \sigma >
  -\frac12, \sigma \leq s \leq \sigma+1$, $\sigma - 2s < -\frac12$.
  \begin{enumerate}
  \item\label{it:trilinear_a} For all $v,u_1,u_2 \in
    \mathcal{S}(\R^3\times \R)$ it holds
    \begin{align}\label{eq:trilinear1}
      \left| I(\mathcal{F}v, \mathcal{F}u_1, \mathcal{F}u_2) \right|
      &\ls{} \|u_1\|_{X^{S}_{-s,\frac12,\infty}}\|u_2\|_{X^{S}_{s,\frac12,\infty}}\|v\|_{X^{W\pm}_{\sigma,\frac12,\infty}},\\
      \label{eq:trilinear2}
      \left| I(\mathcal{F}v, \mathcal{F}u_1, \mathcal{F}u_2) \right| &
      \ls{} \|u_1\|_{X^{S}_{s,\frac12,\infty}}
      \|u_2\|_{X^{S}_{s,\frac12,\infty}}\|v\|_{X^{W\pm}_{-1-\sigma,\frac12,\infty}}.
    \end{align}
  \item\label{it:trilinear_b} There exists $b = b(s,\sigma)
    < \frac12$ in the above regime for $s,\sigma$ such that the above
    inequalities hold true with $\|u_2\|_{X^{S}_{s,b,\infty}}$ instead.
  \end{enumerate}
\end{proposition}

Obviously part i) in Theorem \ref{thm:tri} follows from part i) in the Proposition.  
Part ii) in Theorem \ref{thm:tri} follows from part ii) of the Proposition and the 
following estimate
\begin{equation} \label{gain} \| f \|_{X_{s,b,1}(T)} \ls T^{\frac12-b}
  \| f \|_{X_{s,\frac12,1}(T)}
\end{equation}
whenever $0 \leq b < \frac12$.  Here
$X_{s,b,1}(T)$ stands both for $X_{s,b,1}^S(T)$ and $X_{s,b,1}^{W
  \pm}(T)$.  A proof of \eqref{gain} can be found in
\cite[Section 5]{bejenaru_zakharov_2009}.

The proof of Proposition \ref{prop:trilinear} is given at the end of
this section.  As building blocks we provide a number of preliminary
estimates first. These are concerned with functions which are
dyadically localized in frequency and modulation. In some cases we
additionally differentiate frequencies by their angular separation.

We start this analysis by recalling the well-known bilinear
generalization of the linear $L^4$ Strichartz estimate for the
Schr\"odinger equation in dimension $2$ which is essentially due to
Bourgain \cite[Lemma 111]{bourgain_refinements_1998}. We observe that
a similar estimate is true for a Wave-Schr\"odinger interaction.
\begin{proposition}[Bilinear Strichartz estimates]
  \
  \begin{enumerate}
  \item Let $u_1,u_2\in L^2(\R^4)$ be dyadically Fourier-localized
    such that
    \[
    \supp \mathcal{F} u_i\subset \mathfrak{P}_{N_i}\cap
    \mathfrak{S}_{L_i}
    \]
    for $L_1,L_2\geq 1$, $N_1,N_2\geq 1$. Then the following estimate
    holds:
    \begin{equation}
      \|u_1u_2\|_{L^2(\R^4)}\ls  N_1 N_2^{-\frac12}L_1^\frac12L_2^\frac12\|u_1\|_{L^2}\|u_2\|_{L^2}.\label{eq:str-schr-schr}
    \end{equation}
  \item Let $u,v\in L^2(\R^4)$ be such that
    \[
    \supp \mathcal{F} v \subset C\times \R \cap
    \mathfrak{W}^\pm_{L},\; \supp \mathcal{F} u \subset
    \mathfrak{P}_{N_1}\cap \mathfrak{S}_{L_1}
    \]
    for $L,L_1\geq 1$, $N_1\geq 1$ and a cube $C\subset \R^3$ of
    sidelength $d\geq 1$.  Then the following estimate holds:
    \begin{equation}
      \|uv\|_{L^2(\R^4)}\ls
      \min\{d,N_1\} N_1^{-\frac12}L^\frac12
      L_1^\frac12 \|u\|_{L^2}\|v\|_{L^2}.
      \label{eq:str-wave-schr-gen}
    \end{equation}
    In particular, if
    \[
    \supp \mathcal{F} v \subset \mathfrak{P}_{N} \cap
    \mathfrak{W}^\pm_{L}, \; \supp \mathcal{F} u \subset
    \mathfrak{P}_{N_1}\cap \mathfrak{S}_{L_1}
    \]
    for $L,L_1\geq 1$, $N,N_1\geq 1$, it follows
    \begin{equation}
      \|uv\|_{L^2(\R^4)}\ls \min\{N,N_1\} N_1^{-\frac12}L^\frac12 L_1^\frac12 \|u\|_{L^2}\|v_1\|_{L^2}.\label{eq:str-wave-schr}
    \end{equation}
  \end{enumerate}
  On the left hand side of \eqref{eq:str-schr-schr},
  \eqref{eq:str-wave-schr} and \eqref{eq:str-schr-schr} we may replace
  each function with its complex conjugate.
  \label{prop:bilinear-str} \end{proposition}
\begin{proof}
  As remarked above the estimate \eqref{eq:str-schr-schr} is due to
  Bourgain \cite[Lemma 111]{bourgain_refinements_1998} for two
  dimensions and has been generalized in \cite[Lemma
  3.4]{colliander_global_2007} to higher dimensions\footnote{note that
    the proof of \cite[Lemma 3.4]{colliander_global_2007} applies with
    $\delta=0$ for functions which are dyadically
    Fourier-localized}. It remains to show
  \eqref{eq:str-wave-schr-gen} and \eqref{eq:str-wave-schr}. With
  $f=\mathcal{F}v$ and $g=\mathcal{F}u$ it follows
  \begin{equation*}\left\|\int
      f(\xi_1,\tau_1)g(\xi-\xi_1,\tau-\tau_1)d\xi_1d\tau_1
    \right\|_{L^2_{\xi,\tau}}
    \ls \sup_{\xi,\tau}|E(\xi,\tau)|^\frac12 \|f\|_{L^2}\|g\|_{L^2}
  \end{equation*}
  by the Cauchy-Schwarz inequality, where
  \begin{equation*}
    E(\xi,\tau)=\{(\xi_1,\tau_1)\in \supp f \mid (\xi-\xi_1,\tau-\tau_1)\in \supp g \}\subset \R^4.
  \end{equation*}
  With $\underline{l}=\min\{L,L_1\}$ and $\overline{l}=\max\{L,L_1\}$
  the volume of this set can be estimated as
  \begin{equation*}
    |E(\xi,\tau)|
    \leq  \underline{l} \cdot
    | \{\xi_1\mid | \tau\pm |\xi_1|+|\xi-\xi_1|^2|\ls \overline{l}, \xi_1\in C,|\xi-\xi_1|\sim N_1\}|,
  \end{equation*}
  by Fubini's theorem.  The latter subset of $\R^3$ is contained in a
  cube of sidelength $m$, where $m \sim \min\{d,N_1\}$, so if $N_1=1$
  the estimate follows. If $N_1\geq 2$ and one component $\xi_{1,i}, i
  \in \{1,2,3\}$ is fixed, then the other two components $\xi_{1,j}, j
  \ne i$ are confined to an interval of length $m$. For each $i \in
  \{1,2,3\}$, we notice that in the subset where $|(\xi-\xi_{1})_i|\gs
  N_1$ we have that
  $|\partial_{\xi_{1,i}}(\tau\pm|\xi_1|+|\xi-\xi_1|^2)|\gs N_1$. This
  shows that
  \[
  | \{\xi_1\mid | \tau\pm |\xi_1|+|\xi-\xi_1|^2|\ls \overline{l},
  \xi_1 \in C,|\xi-\xi_1|\sim N_1\}|\ls N_1^{-1}\overline{l}m^2,
  \]
  and the claim \eqref{eq:str-wave-schr-gen} follows. This also
  implies the claim \eqref{eq:str-wave-schr} because the dyadic
  annulus of radius $N$ is contained in a cube of sidelength $d\sim
  N$.
\end{proof}

\begin{proposition}[Transverse high-high interactions, low
  modulation]\label{prop:trans_low_mod}
  Let $f,g_1,g_2\in L^2$ with $\| f \|_{L^2} =\| g_1\|_{L^2}=
  \|g_2\|_{L^2}=1$ and
  \begin{equation*}
    \supp(f)\subset \mathfrak{P}_{N}\cap \mathfrak{W}^\pm_L,\quad 
    \supp(g_k)\subset \mathfrak{Q}_A^{j_k} \cap \mathfrak{P}_{N_k} \cap \mathfrak{S}_{L_k}\quad (k=1,2).
  \end{equation*}
  where the frequencies $N,N_1,N_2$ and modulations $L,L_1,L_2$
  satisfy
  \[
  1 \ll N \ls N_1 \sim N_2, \qquad L_1,L_2,L \ls N_1^2
  \]
  while the angular localization parameters $A$ and $j_1, j_2 \in
  \Omega_A$ satisfy
  \[
  1\leq A \ll N_1, \qquad \alpha(j_1,j_2)\sim A^{-1}.
  \]
  Then the following estimate holds
  \begin{equation} \label{eq:trans_low_mod} |I(f,g_1,g_2)| \ls
    N_1^{-\frac12} (L_1 L_2 L)^\frac12.
  \end{equation}
\end{proposition}

\begin{proof} We abuse notation and replace $g_2$ by $g_2(-\cdot)$ and
  change variables $\zeta_2\mapsto -\zeta_2$ to obtain the usual
  convolution structure. From now on it holds $|\tau_2-|\xi_2|^2|\sim
  L_2$ within the support of $g_2$. We consider only the case
  $\supp(f)\subset \mathfrak{W}^-_L$ since in the case
  $\supp(f)\subset \mathfrak{W}^+_L$ the same arguments apply.

  For fixed $\xi_1,\xi_2$ we change variables $c_1=\tau_1+|\xi_1|^2$,
  $c_2=\tau_2-|\xi_2|^2$.  By decomposing $f$ into $\sim L$ pieces and
  applying the Cauchy-Schwarz inequality, it suffices to prove
  \begin{equation}\label{eq:reduce}
    \begin{split}
      &\left| \int g_{1}(\phi^-_{c_1}(\xi_1)) g_{2}
        (\phi^+_{c_2}(\xi_2))
        f(\phi^-_{c_1}(\xi_1)+\phi^+_{c_2}(\xi_2)) d\xi_1 d\xi_2 \right| \\
      \ls{} & N_1^{-\frac12} \| g_{1}\circ\phi^-_{c_1}\|_{L^2_\xi} \|
      g_{2}\circ \phi^+_{c_2} \|_{L^2_\xi} \| f\|_{L^2}
    \end{split}
  \end{equation}
  where $f$ is now supported in $c \leq \tau-|\xi| \leq c+1$ and
  $\phi^\pm_{c_k}(\xi)=(\xi,\pm|\xi|^2 + c_k)$, $k=1,2$, and the
  implicit constant is independent of $c, c_1,c_2$.

  We refine the localization of the $\xi$ and $\tau$ components by
  orthogonality methods.  Since the support of $f$ in the $\tau$
  direction is confined to an interval of length $\ls N_1$,
  $|\xi_2|^2-|\xi_1|^2$ is localized in an interval of length $\sim
  N_1$ which in turn localizes $|\xi_2|-|\xi_1|$ in an interval of
  size $\sim 1$. By decomposing the plane into annuli of size $\sim 1$
  and using the Cauchy-Schwarz inequality, we reduce \eqref{eq:reduce}
  further to the additional assumption that $|\xi_1|$ and $|\xi_2|$
  are localized in two intervals of length $\sim 1 \ls N_1 A^{-1}$.
  Recalling the additional angular localization, we can assume that
  $g_1,g_2$ and $f$ are each localized in cubes of size $N_1 A^{-1}$
  with respect to the $\xi$ variables.
        
  We use the parabolic scaling $(\xi,\tau)\mapsto (N_1 \xi, N_1^2
  \tau)$ to define
  \begin{equation*}
    \tilde{f}(\xi,\tau)=f(N_1 \xi,N_1^2\tau), \;
    \tilde{g}_k(\xi_k,\tau_k)=g_{k}(N_1 \xi_k,N_1^2\tau_k),\; k=1,2.
  \end{equation*}
  If we set $\tilde{c}_k=c_kN_k^{-2}$, the estimate \eqref{eq:reduce}
  reduces to
  \begin{equation} \label{eq:reduce1}
    \begin{split}
      &\left| \int \tilde{g}_{1}(\phi^-_{\tilde{c}_1}(\xi_1))
        \tilde{g}_{2} (\phi^+_{\tilde{c}_2}(\xi_2))
        \tilde{f}(\phi^-_{\tilde{c}_1}(\xi_1)+\phi^+_{\tilde{c}_2}(\xi_2)) d\xi_1 d\xi_2 \right| \\
      \ls{}& N_1^{-1} \| \tilde{g}_{1}\circ\phi^-_{\tilde{c}_1}
      \|_{L^2_\xi} \| \tilde{g}_{2} \circ\phi^+_{\tilde{c}_2}
      \|_{L^2_\xi} \| \tilde{f} \|_{L^2},
    \end{split}
  \end{equation}                
  where now $\tilde{g}_k$ is supported in a cube of size $\sim A^{-1}$
  with $|\xi_1|,|\xi_2| \sim 1$ and the supports are separated by
  $\sim A^{-1}$.  $\tilde{f}$ is supported in a neighborhood of size
  $N^{-2}_1$ of the submanifold $S_3$ parametrized by
  $(\xi,\psi_{N_1}(\xi))$ for $\psi_{N_1}(\xi)=\frac{|\xi|}{N_1} +
  \frac{c}{N^2_1}$.  Let us put $\varepsilon=N_1^{-2}$ and denote this
  neighborhood by $S_3(\varepsilon)$.  The separation of $\xi_1$ and
  $\xi_2$ above implies also that in the support of $\tilde{f}$ we
  have $|\xi| \gs A^{-1} \geq N_1^{-1}$.

  By density and duality it is enough to consider continuous $\tilde
  g_1,\tilde g_2$ and we can further rewrite the above estimate as
  \begin{equation} \label{submanifold} \| \tilde g_1|_{S_1} \ast
    \tilde g_2|_{S_2} \|_{L^2(S_3(\varepsilon))} \ls
    \varepsilon^\frac12 \| \tilde g_1 \|_{L^2(S_1)} \| \tilde g_2
    \|_{L^2(S_2)}
  \end{equation}
  where $S_i$, $i=1,2$ are parametrized by $\phi^\pm_{\tilde c_i}$.
  The above localization properties of the support of $\tilde{g}_i$
  are inherited by $S_i$, which implies that the maximal diameter of
  the $S_1$, $S_2$ and $S_3$ is at most $R\sim A^{-1}$.
  
  It is a straightforward to check that $S_1$, $S_2$ and $S_3$ verify
  Part \ref{it:reg_cond} of Assumption \ref{assumpt_main} with
  $\beta=1$ and $b \sim 1$.
 
  We now turn our attention to the transversality condition, i.e. Part
  \ref{it:trans} of Assumption \ref{assumpt_main}.  Since
  $\alpha(j_1,j_2)\sim A^{-1}$, there exists a unit vector $v$ which
  is almost orthogonal to any $\xi_1 \in \mathfrak{Q}_A^{j_1}$ and any
  $\xi_2 \in \mathfrak{Q}_A^{j_2}$ in the following sense
  \begin{equation} \label{vol} |\det
    \left(\frac{\xi_1}{|\xi_1|},\frac{\xi_2}{|\xi_2|}, v
    \right)|=\mathrm{vol}
    \left(\frac{\xi_1}{|\xi_1|},\frac{\xi_2}{|\xi_2|}, v \right) \sim
    |\sin \angle{(\frac{\xi_1}{|\xi_1|},\frac{\xi_2}{|\xi_2|})}| \sim
    A^{-1}
  \end{equation}
  The codimensions of $S_1, S_2, S_3$ add up to $3$ instead of $4$.
  In order to be able to apply the results in the first part of the
  paper, we foliate one of the surfaces to increase its codimension by
  one. We do this for $S_3$ as follows:
  \[
  S_3 = \bigcup_{c \in I} S_3^c
  \]
  where $S_3^c=S_3 \cap \{ (\xi,\tau): \xi \cdot v= c \}$ and $c$
  varies in an interval $I$ of length $|I|\sim A^{-1}$. Each $S_3^c$
  retains its $C^{1,1}$ structure. In addition,
  \begin{equation} \label{integral} \| f \|_{L^2(S_3)}^2 = \int_I \| f
    \|_{L^2(S_3^c)}^2 \ dc \lesssim A^{-1} \sup_c \| f \|_{
      L^2(S_3^c)}^2
  \end{equation}
  For fixed $c\in I$, let us identify a basis of unit normals to
  $S_3^c$.  For the following calculations, we set $\xi=\xi_0$ and
  denote the components as
  \[
  \xi_k=(\xi_{k,1},\xi_{k,2},\xi_{k,3}), \, k=0,1,2.
  \]
  At each point we keep the normal to the cone
  \[
  \mathfrak{n}_{S_3} =(\frac{\xi_{0,1}}{|\xi| \langle N_1\rangle},
  \frac{\xi_{0,2}}{|\xi| \langle N_1\rangle}, \frac{\xi_{0,3}}{|\xi|
    \langle N_1\rangle}, - \frac{N_1}{ \langle N_1\rangle}).
  \]
  Another convenient normal is $\mathfrak{n}_{S_3^c}=(v,0)$.  This
  choice is simple, but it has the disadvantage that
  $\{\mathfrak{n}_{S_3}, \mathfrak{n}_{S_3^c}\}$ is not an orthonormal
  basis. On the other hand,
  \begin{equation} \label{estnor} |\mathfrak{n}_{S_3}\cdot
    \mathfrak{n}_{S_3^c}| =| v\cdot (\frac{\xi_{0,1}}{|\xi| \langle
      N_1\rangle}, \frac{\xi_{0,2}}{|\xi| \langle N_1\rangle},
    \frac{\xi_{0,3}}{|\xi| \langle N_1\rangle})| \leq \frac{1}{\langle
      N_1\rangle} \ll 1.
  \end{equation}
  Therefore, a correct orthonormal set of normals to $S_3^c$ is
  $\{\mathfrak{n}_{S_3}, \mathfrak{n}_{S_3}'\}$, with
  \[
  \mathfrak{n}_{S_3}'= \frac{\mathfrak{n}_{S_3^c}-
    (\mathfrak{n}_{S_3}\cdot \mathfrak{n}_{S_3^c})
    \mathfrak{n}_{S_3}}{|\mathfrak{n}_{S_3^c}-(\mathfrak{n}_{S_3}\cdot
    \mathfrak{n}_{S_3^c}) \mathfrak{n}_{S_3} |}.
  \]
  Now, we can analyze the transversality properties of our
  submanifolds $S_1$, $S_2$, $S_3^c$ in the sense of \eqref{eq:trans}.
  Let $\mathfrak{n}_1, \mathfrak{n}_2$ be the unit normals at $S_1$,
  respectively $S_2$.  Then we need to determine the absolute value of
  the determinant
  \[
  d = \det (\mathfrak{n_1}, \mathfrak{n}_2, \mathfrak{n}_{S_3} ,
  \mathfrak{n}_{S_3}') = \frac{1}{|\mathfrak{n}_{S_3^c}-(
    \mathfrak{n}_{S_3}\cdot \mathfrak{n}_{S_3^c}) \mathfrak{n}_{S_3}
    |} \det (\mathfrak{n_1}, \mathfrak{n}_2, \mathfrak{n}_{S_3} ,
  \mathfrak{n}_{S_3^c}).
  \]
  In view of \eqref{estnor} we obtain
  \begin{equation*}
    d\sim \det (\mathfrak{n_1}, \mathfrak{n}_2, \mathfrak{n}_{S_3} ,
    \mathfrak{n}_{S_3^c})=\begin{vmatrix}
      \frac{2\xi_{1,1}}{\la 2\xi_1 \ra} & \frac{2\xi_{2,1}}{\la 2\xi_2 \ra} & \frac{\xi_{0,1}}{|\xi_0| \la N_1 \ra} & v_1 \\
      \frac{2\xi_{1,2}}{\la 2\xi_1 \ra} & \frac{2\xi_{2,2}}{\la 2\xi_2 \ra} & \frac{\xi_{0,2}}{|\xi_0| \la N_1 \ra} & v_2 \\
      \frac{2\xi_{1,3}}{\la 2\xi_1 \ra} & \frac{2\xi_{2,3}}{\la 2\xi_2 \ra} & \frac{\xi_{0,3}}{|\xi_0| \la N_1 \ra} & v_3 \\
      \frac{1}{\la 2\xi_1 \ra} & -\frac{1}{\la 2\xi_2 \ra} & -\frac{N_1}{\la N_1 \ra} & 0 \\
    \end{vmatrix}.
  \end{equation*}
  Expansion along the third column shows that
  \begin{equation*}
    |\det (\mathfrak{n_1}, \mathfrak{n}_2, \mathfrak{n}_{S_3} ,
    \mathfrak{n}_{S_3^c})-\tilde{d}|\ls N_1^{-1},
  \end{equation*}
  i.e. the main contribution comes from the $(4,3)$-minor
  \begin{equation*}
    \tilde{d}= \frac{N_1}{\la N_1 \ra} \begin{vmatrix} \frac{2\xi_{1,1}}{\la 2\xi_1 \ra} & \frac{2\xi_{2,1}}{\la 2\xi_2 \ra} & v_1 \\
      \frac{2\xi_{1,2}}{\la 2\xi_1 \ra} & \frac{2\xi_{2,2}}{\la 2\xi_2 \ra} & v_2 \\
      \frac{2\xi_{1,3}}{\la 2\xi_1 \ra} & \frac{2\xi_{2,3}}{\la 2\xi_2 \ra} & v_3
    \end{vmatrix},
  \end{equation*}
  which can be rewritten as
  \[
  \tilde{d}= \frac{N_1}{\la N_1 \ra} \frac{2|\xi_1|}{\la 2\xi_1 \ra}
  \frac{2|\xi_2|}{\la 2\xi_2\ra} \begin{vmatrix} \frac{\xi_{1,1}}{|\xi_1|} & \frac{\xi_{2,1}}{|\xi_2|} & v_1 \\
    \frac{\xi_{1,2}}{|\xi_1|} & \frac{\xi_{2,2}}{|\xi_2|} & v_2 \\
    \frac{\xi_{1,3}}{|\xi_1|} & \frac{\xi_{2,3}}{|\xi_2|} & v_3
  \end{vmatrix}
  =\frac{N_1}{\la N_1 \ra} \frac{2|\xi_1|}{\la 2\xi_1 \ra}
  \frac{2|\xi_2|}{\la 2\xi_2 \ra} \det
  \left(\frac{\xi_1}{|\xi_1|},\frac{\xi_2}{|\xi_2|}, v \right)
  \]
  which, by \eqref{vol}, implies that $|\tilde{d}| \sim A^{-1}\gg
  N_1^{-1}$.  Therefore we have established that $|d| \sim A^{-1}$.
  Recalling that the diameters of $S_1, S_2, S_3^c$ are $\sim A^{-1}$,
  we can now apply Corollary \ref{cor:conv_theta} which implies
  \[
  \| \tilde g_1|_{S_1} \ast \tilde g_2|_{S_2} \|_{L^2(S_3^c)} \ls
  A^{\frac12} \| \tilde g_1 \|_{L^2(S_1)} \| \tilde g_2 \|_{L^2(S_2)}.
  \]
  From \eqref{integral} we obtain
  \[
  \| \tilde g_1|_{S_1} \ast \tilde g_2|_{S_2} \|_{L^2(S_3)} \ls \|
  \tilde g_1 \|_{L^2(S_1)} \| \tilde g_2 \|_{L^2(S_2)}
  \]
  and \eqref{submanifold} follows.
\end{proof}

\begin{proposition}[Parallel high-high interactions]\label{prop:paral_hh}
  Let $f,g_1,g_2\in L^2$, $\|f\|_{L^2}=\|g_1\|_{L^2}=\|g_2\|_{L^2}=1$
  such that
  \begin{equation*}
    \supp(f)\subset \mathfrak{P}_{N}\cap \mathfrak{W}^\pm_L ,\quad 
    \supp(g_k)\subset \mathfrak{Q}^{j_k}_A \cap \mathfrak{P}_{N_k} \cap \mathfrak{S}_{L_k}\quad (k=1,2),
  \end{equation*}
  with $1 \ll N \ls N_1 \sim N_2$. Assume that $A \sim N_1$ and
  $\alpha(j_1,j_2)\ls A^{-1}$.  Then for all $L,L_1,L_2\geq 1$ we have
  \begin{equation}\label{eq:paral_hh}
    |I(f,g_1,g_2)| \ls N_1^{-\frac12} (L_{1} L_{2} L)^{\frac12}.
  \end{equation}
\end{proposition}
\begin{proof}
  After a rotation we may assume that the angular localization is such
  that the first spherical cap $\omega_A^{j_1}$ is centered at
  $(1,0,0)$ and the second spherical cap $\omega_A^{j_2}$ is located
  at distance $\ls A^{-1}$ from $(\pm 1, 0,0)$.  Then, if
  $(\xi_k,\tau_k)\in \supp g_k$ for $k=1,2$, and $\xi_0=\xi_1+\xi_2\in
  \supp f$ we have
  \begin{equation} \label{xii}
    |\xi_{1,2}|+|\xi_{1,3}|+|\xi_{2,2}|+|\xi_{2,3}|+|\xi_{0,2}|+|\xi_{0,3}|\ls
    1.
  \end{equation}
  This shows that $|\xi_{1,1} +\xi_{2,1}|=|\xi_{0,1}|\sim N$,
  $|\xi_{1,1}|,|\xi_{2,1}| \sim N_1$.

  In the following, we use almost orthogonality methods to further
  localize all functions to smaller pieces, for which the claim is
  trivial.

  By decomposing $f,g_1,g_2$ into $\sim L,L_1,L_2$ pieces,
  respectively, and applying the Cauchy-Schwarz inequality, it
  suffices to prove
  \begin{equation}\label{eq:reduce2}
    \begin{split}
      &\left| \int g_{1}(\xi_1,\tau_1) g_{2} (\xi_2,\tau_2)
        f(\xi_1+\xi_2,\tau_1+\tau_2) d\xi_1 d\xi_2 d\tau_1 d\tau_2 \right| \\
      \ls{} & N_1^{-\frac12} \| g_{1}\|_{L^2} \| g_{2}\|_{L^2} \|
      f\|_{L^2},
    \end{split}
  \end{equation}
  where $f$ is now supported in $c \leq \tau-|\xi| \leq c+1$ and
  $g_{k}$ is supported in $c_k \leq \tau_k-|\xi_k|^2 \leq
  c_k+1$. Therefore, with respect to the $\tau$ variable, $f$ is
  supported in an interval of length $\sim N$. Using orthogonality, we
  can further localize $g_k$ with respect to the second variable
  $\tau_k$ to intervals of length $\sim N$, $k=1,2$. In turn this
  implies that the spatial frequencies $\xi_k$ can be localized
  further to annuli of width $\sim N N_1^{-1} \leq 1$. In light of
  \eqref{xii} we can strengthen the localization of $g_k$ with respect
  to $\xi_k$ to cubes of side-length $\sim 1$. As a consequence, we
  also improve the localization of the $\xi$-support of $f$ to cubes
  of size $\sim 1$, which then also allows to localize $f$ with
  respect to $\tau$ to intervals of length $\sim 1$.  Now, we repeat
  the above procedure: We can further localize $g_k$ with respect to
  $\tau_k$ to intervals of length $\sim 1$, which also implies a
  better localization for $g_k$ with respect to $\xi_k$ to annuli of
  width $\sim N_1^{-1}$.

  In summary, we have reduced the problem to the case when the volume
  of the supports of $g_1$ and $g_2$ is $\sim N_1^{-1}$ which then
  trivially gives \eqref{eq:reduce2} by virtue of the Cauchy-Schwarz
  inequality.
\end{proof}

Next, we summarize the previous two results in the following
Corollary, which settles the high-high to low interactions with low
modulation.
\begin{corollary}[high-high to low interactions, low
modulation]\label{cor:hhl_lm}
  Assume that $f,g_1,g_2\in L^2$ with $\| f \|_{L^2} =\| g_1\|_{L^2}=
  \|g_2\|_{L^2}=1$ and
  \begin{equation*}
    \supp(f)\subset  \mathfrak{P}_{N}\cap \mathfrak{W}^\pm_L,\quad 
    \supp(g_k)\subset \mathfrak{P}_{N_k} \cap \mathfrak{S}_{L_k}\quad (k=1,2),
  \end{equation*}
  where $N,N_1,N_2$ and $L,L_1,L_2$ satisfy
  \[
  1\ll N\ls N_1 \sim N_2, \qquad L_1,L_2,L \ls N_1^2.
  \]
  Then, the following estimate holds
  \begin{equation} \label{eq2:trans_low_mod} |I(f,g_1,g_2)| \ls (L_1
    L_2 L)^\frac12 N_1^{-\frac12} \log{N_1}.
  \end{equation}  
\end{corollary}

\begin{proof}
  It suffices to consider non-negative $f,g_1,g_2$. We choose a
  threshold $M=C N_1$ such that for $A< M$ Proposition
  \ref{prop:trans_low_mod} respectively for $A=M$
  Proposition \ref{prop:paral_hh} is applicable, and decompose
  \begin{align*}
    | I(f,g_1, g_2)| \leq &\sum_{A=1}^{M-1} \sum_{\alpha(j_1, j_2) \sim
      A^{-1}} I(f,Q_A^{j_1} g_1, Q_A^{j_2} g_2) \\
    &{}+ \sum_{\alpha(j_1,j_2) \ls M^{-1}}
    I(f,Q_M^{j_1} g_1, Q_M^{j_2} g_2).
  \end{align*}
  Concerning the first sum, we use \eqref{eq:trans_low_mod} for fixed
  $A$ and obtain
  \begin{align*} \sum_{\alpha(j_1, j_2) \sim A^{-1}} I(f,Q_A^{j_1}
    g_1, Q_A^{j_2} g_2) \ls & \frac{(L_1 L_2L)^\frac12}{N_1^{\frac12}}
    \sum_{\alpha(j_1, j_2) \sim
      A^{-1}}\|Q_A^{j_1}g_1\|_{L^2}\|Q_A^{j_2}g_2\|_{L^2}\\
    \ls & N_1^{-\frac12} (L_1 L_2L)^\frac12,
  \end{align*}
  where we use Cauchy-Schwarz in the last step.  Concerning the second
  sum, we use \eqref{eq:paral_hh} for fixed $A$ and obtain the same
  bound as above.

  Dyadic summation with respect to $A$ introduces the additional factor
  $\log{N_1}$, which leads to \eqref{eq2:trans_low_mod}.
\end{proof}

The case of high-high to low interactions with high modulation is
covered by the following Proposition.
\begin{proposition}[high-high to low interactions, high modulation]\label{prop:hhl_hm}
  Assume that $f,g_1,g_2\in L^2$ with $\| f \|_{L^2} =\| g_1\|_{L^2}=
  \|g_2\|_{L^2}=1$ and
  \begin{equation*}
    \supp(f)\subset  \mathfrak{P}_{N}\cap \mathfrak{W}^\pm_L,\quad 
    \supp(g_k)\subset \mathfrak{P}_{N_k} \cap \mathfrak{S}_{L_k}\quad (k=1,2),
  \end{equation*}
  where $N,N_1,N_2$ and $L,L_1,L_2$ satisfy
  \[
  1 \leq N\ls N_1 \sim N_2, \qquad N_1^2 \ls \max\{L,L_1,L_2\}.
  \]
  Then, the following estimate holds
  \begin{equation} \label{eq:high_mod} |I(f,g_1,g_2)| \ls
    N_1^{-\frac12} (L_1 L_2 L)^\frac12 (\max\{L,L_1,L_2\}
    N_1^{-2})^{-\frac12}.
  \end{equation}  
\end{proposition}

\begin{proof} \emph{Case a)} $L=\max\{L,L_1,L_2\}$: Use Cauchy-Schwarz
  and \eqref{eq:str-schr-schr}.

  \emph{Case b)} $L_1 = \max\{L,L_1,L_2\}$ or $L_2 =
  \max\{L,L_1,L_2\}$: Use Cauchy-Schwarz and \eqref{eq:str-wave-schr}.
\end{proof}

The next proposition covers the case of low-high interactions.
\begin{proposition}[low-high interactions]\label{prop:high-low}
  Let $f,g_1,g_2\in L^2$ be functions with
  $\|f\|_{L^2}=\|g_1\|_{L^2}=\|g_2\|_{L^2}=1$ such that
  \begin{equation*}
    \supp(f)\subset \mathfrak{P}_{N}\cap \mathfrak{W}^\pm_L ,\quad 
    \supp(g_k)\subset \mathfrak{P}_{N_k} \cap \mathfrak{S}_{L_k}\quad (k=1,2),
  \end{equation*}
  with $1\leq N_1\ll N_2$.
\begin{enumerate}
\item\label{it:l2_low}
  If $L_2 \ll N_2^{2}$, then we have
  \begin{equation}\label{eq:high-low}
    |I(f,g_1,g_2)| \ls  N_1 N_2^{-\frac12} (L_{1} L_{2} L)^{\frac12}
    \max{(L,L_1,L_2)}^{-\frac12}.
  \end{equation}
\item\label{it:l2_high}
  If $L_2 \gs N_2^{2}$, then we have
  \begin{equation}\label{eq:high-low2}
    |I(f,g_1,g_2)| \ls  N_1^{\frac12}\min\{L,L_1\}^{\frac12}\min\{N_1^2,\max\{L,L_1\}\}^{\frac12}.
\end{equation}
\end{enumerate}
\end{proposition}
\begin{proof}
  The integral vanishes unless $N_2 \sim N$ and
  \begin{equation}\label{eq:res_in}
    \max\{L,L_1,L_2\}\gs ||\xi_1|^2-|\xi_2|^2\pm
    |\xi_1-\xi_2||\gs N_2^2.
  \end{equation}
  We split the proof into two cases:

  {\it Case a)} $L_2 \ll N_2^{2}$.

  {\it Subcase i)} $L=\max\{L,L_1,L_2\}$. The bilinear $L^2$ estimate
  \eqref{eq:str-schr-schr} yields
  \begin{align*}
    |I(f,g_1,g_2)| \ls
    \|f\|_{L^2}\|\mathcal{F}^{-1}g_1\overline{\mathcal{F}^{-1}g_2}\|_{L^2}
    \ls L_1^\frac12 L_2^\frac12 N_1 N_2^{-\frac12}.
  \end{align*}

  {\it Subcase ii)} $L_1=\max\{L,L_1,L_2\}$. Since $g_1$ is localized
   to a cube of sidelength $N_1$ with respect to the $\xi_1$ variable, by almost
  orthogonality the estimate reduces to the case when $f$ and $g_2$
  are similarly localized to cubes of sidelength $N_1$.
  Then we use the bilinear $L^2$ estimate \eqref{eq:str-wave-schr-gen}
  with $d=N_1$ to obtain
  \begin{align*}
    |I(f,g_1,g_2)| \ls \|g_1\|_{L^2}\|\mathcal{F}^{-1}
    f\mathcal{F}^{-1} g_2\|_{L^2} \ls L^\frac12 L_2^\frac12 N_1
    N_2^{-\frac12}
  \end{align*}
  This finishes the proof of \eqref{eq:high-low}.

  {\it Case b)} $L_2\gs N_2^2$. Again, since $g_1$ is localized
  to a cube of sidelength $N_1$ with respect to the $\xi$ variable, the estimate
  reduces to the case when $f$ and $g_2$ are
  localized to cubes of sidelength $N_1$ with respect to the $\xi$ variables.

  {\it Subcase i)} $L\leq L_1$ and $N_1^2\leq \max\{L,L_1\}$. The volume of the support of $f$ is
  $\ls N_1^3L$, and we estimate
  \[
  |I(f,g_1,g_2)| \ls \| f \|_{L^1} \|g_1\|_{L^2} \|g_2\|_{L^2} \ls
  N_1^\frac32 L^\frac12.
  \]

  {\it Subcase ii)} $L_1< L$ and $N_1^2\leq \max\{L,L_1\}$. The volume of the support of $g_1$ is
  $\ls N_1^3L_1$, and we estimate
  \[
  |I(f,g_1,g_2)| \ls \| f \|_{L^2} \|g_1\|_{L^1} \|g_2\|_{L^2} \ls
  N_1^\frac32 L_1^\frac12.
  \]

  {\it Subcase iii)} $N_1^2> \max\{L,L_1\}$. Cauchy-Schwarz and the bilinear $L^2$ estimate \eqref{eq:str-wave-schr-gen} yields
\[
  |I(f,g_1,g_2)| \ls \|\mathcal{F}^{-1}
    f\mathcal{F}^{-1} g_1\|_{L^2} \|g_2\|_{L^2}\ls N_1^{\frac12} L^\frac12 L_1^\frac12,
  \]
which finishes the proof of \eqref{eq:high-low2}.
\end{proof}

Finally, we deal with the case where the wave frequency is very small.
\begin{proposition}[very small wave frequency]\label{prop:very_small_wave_freq}
  Assume that $f,g_1,g_2\in L^2$ with
  $\|f\|_{L^2}=\|g_1\|_{L^2}=\|g_2\|_{L^2}=1$ such that
  \begin{equation*}
    \supp(f)\subset \mathfrak{P}_{N}\cap \mathfrak{W}^\pm_L ,\quad 
    \supp(g_k)\subset \mathfrak{P}_{N_k} \cap \mathfrak{S}_{L_k}\quad (k=1,2),
  \end{equation*}
  and assume that $N\ls 1$. Then,
  \begin{equation}\label{eq:very_small_wave_freq}
    |I(f,g_1,g_2)|  
    \ls \min{(L,L_1,L_2)}^{\frac12}.
  \end{equation}
\end{proposition}
\begin{proof} Using orthogonality we reduce the problem to the case
  when both $g_1$ and $g_2$ are supported in cubes of size $\sim 1$
  with respect to the $\xi_k$ variables.  Then, the volume of the support
  of $f$ is $L$, while the volume of the support of $g_k$ is $L_k$.
  If $L=\min{(L,L_1,L_2)}$, then by using the trivial estimate
  
  \[
  |I(f,g_1,g_2)| \lesssim \| f \|_{L^1} \| g_1 \|_{L^2} \| g_2
  \|_{L^2} \ls L^\frac12
  \]
  the claim follows. If $L_1 = \min{(L,L_1,L_2)}$ then we obtain
  
  \[
  |I(f,g_1,g_2)| \lesssim \| f \|_{L^2} \| g_1 \|_{L^1} \| g_2
  \|_{L^2} \ls L_1^\frac12.
  \]
  The case $L_2 = \min{(L,L_1,L_2)}$ follows in a similar manner.
\end{proof}

\subsection{Proof of Proposition \ref{prop:trilinear}}
We prove parts \ref{it:trilinear_a} and \ref{it:trilinear_b} at the
same time.  We focus on
establishing \eqref{eq:trilinear1} and \eqref{eq:trilinear2} as stated
in part \ref{it:trilinear_a}. Then, at any step we show that we can
improve the corresponding estimate by using the $X_{s,b,\infty}$ norm
instead of $X_{s,\frac12,\infty}$ norm on of the terms involved in
the estimate, where $b$ is a parameter which depends on $s$ and
$\sigma$. The conditions on $b$ will accumulate in several
steps but one has to keep in mind that $b < \frac12$ is the starting
condition and it will not be repeated.

\begin{proof}[Proof of Proposition \ref{prop:trilinear}]
By definition of the norms it is enough to consider functions with
non-negative Fourier transform. We dyadically decompose
\begin{equation*}
  u_k=\sum_{N_k,L_k\geq 1} S_{L_k}P_{N_k}u_k
  \; ,\quad v=\sum_{N,L\geq 1 } W^\pm_{L}P_{N} v.
\end{equation*}
Setting $g^{L_k,N_k}_k=\mathcal{F}S_{L_k}P_{N_k}u_k$ and
$f^{L,N}=\mathcal{F}W^\pm_{L}P_{N} v$, we observe
\[
  I( \mathcal{F}v, \mathcal{F}u_1, \mathcal{F}u_2) =
  \sum_{N,N_1,N_2\geq 1} \sum_{L,L_1,L_2\geq 1} I( f^{L,N},
  g_1^{L_1,N_1}, g_2^{L_2,N_2}).
\]
{\it Case a)} high-high-low interactions, i.e. $N_1\sim N_2\gs
N\gg 1$.  Using \eqref{eq2:trans_low_mod} and
\eqref{eq:high_mod} it follows that
\begin{align*}
  & \sum_{L,L_1,L_2\geq 1} |I( f^{L,N}, g_1^{L_1,N_1},g_2^{L_2,N_2})| \\
  \ls {}& N_1^{-\frac12} \log{N_1} \sum_{L,L_1,L_2\leq
    N_1^2}
  L^{\frac12}\|f^{L,N} \|_{L^2} L_1^{\frac12}\|g_1^{L_1,N_1}\|_{L^2}   L_2^{\frac12}\|g_2^{L_2,N_2}\|_{L^2} \\
  & + N_1^{-\frac12} \sum_{\max\{L,L_1,L_2\}> N_1^2} \frac{N_1}{\max\{L,L_1,L_2\}^\frac12} L^{\frac12}\|f^{L,N} \|_{L^2} L_1^{\frac12}\|g_1^{L_1,N_1}\|_{L^2}   L_2^{\frac12}\|g_2^{L_2,N_2}\|_{L^2}  \\
  \ls {}& N_1^{-\frac12} (\log{N_1})^4 \|P_Nv
  \|_{X^{W\pm}_{0,\frac12,\infty}}
  \|P_{N_1}u_1\|_{X^S_{0,\frac12,\infty}}
  \|P_{N_2}u_2\|_{X^S_{0,\frac12,\infty}}.
\end{align*}
A straightforward modification also shows the bound
\begin{align*}
  \ls & N_1^{\frac12-2b}(\log{N_1})^4
  \|P_Nv\|_{X^{W\pm}_{0,\frac12,\infty}}
  \|P_{N_1}u_1\|_{X^S_{0,\frac12,\infty}}
  \|P_{N_2}u_2\|_{X^S_{0,b,\infty}}
\end{align*}
In order to prove \eqref{eq:trilinear1} we perform the summation with respect to $1 \ll N\leq N_1 \sim N_2$ and obtain
  \begin{align*}
    & \sum_{1 \ll N\leq  N_1 \sim N_2} N ^{-\sigma} N_1^{-\frac12} (\log{N_1})^4  \|P_Nv \|_{X^{W\pm}_{\sigma,\frac12,\infty}} \|P_{N_1}u_1\|_{X^S_{-s,\frac12,\infty}}  \|P_{N_2}u_2\|_{X^S_{s,\frac12,\infty}} \\
    \ls {} &  \|v\|_{X^{W\pm}_{\sigma,\frac12,\infty}} \sum_{1 \ll  N_1 \sim N_2}   \|P_{N_1}u_1\|_{X^S_{-s,\frac12,\infty}}  \|P_{N_2}u_2\|_{X^S_{s,\frac12,\infty}} \\
     \ls{} &\|v\|_{X^{W\pm}_{\sigma,\frac12,\infty}}
    \|u_1\|_{X^S_{-s,\frac12,\infty}} \|u_2\|_{X^S_{s,\frac12,\infty}},
  \end{align*}
  where we have used that $\sigma > -\frac12$.
  If we choose $b$ such that $\frac12-2b-\sigma< 0$, then we also obtain
 \[
 \sum_{1 \ll N\leq N_1 \sim N_2 } |I(f^{N}, g_1^{N_1}, g_2^{N_2})| \ls
 \|v\|_{X^{W\pm}_{\sigma,\frac12,\infty}}
 \|u_1\|_{X^S_{-s,\frac12,\infty}} \|u_2\|_{X^S_{s,b,\infty}}.
 \]

 For proving \eqref{eq:trilinear2} in this case, we perform the summation as follows:
 \begin{align*}
   &  \sum_{1 \ll N\leq  N_1 \sim N_2} (\frac{N}{N_1})^{1+\sigma} N_1^{\frac12+\sigma-2s} (\log{N_1})^4  \|P_N v \|_{X^{W\pm}_{-1-\sigma,\frac12,\infty}} \|P_{N_1}u_1\|_{X^S_{s,\frac12,\infty}}  \|P_{N_2}u_2\|_{X^S_{s,\frac12,\infty}} \\
    \ls {}&  \|v\|_{X^{W\pm}_{-1-\sigma,\frac12,\infty}} \sum_{1 \ll  N_1 \sim N_2}     \|P_{N_1}u_1\|_{X^S_{s,\frac12,\infty}}  \|P_{N_2}u_2\|_{X^S_{s,\frac12,\infty}}\\
   \ls {}& \|v\|_{X^{W\pm}_{-1-\sigma,\frac12,\infty}} \|P_{N_1}u_1\|_{X^S_{s,\frac12,\infty}}  \|P_{N_2}u_2\|_{X^S_{s,\frac12,\infty}},
 \end{align*}
 where we have used that $\sigma - 2s < - \frac12$.  By picking $b$
 such that $\frac32-2b+\sigma-2s < 0$, we also obtain
 \[
 \sum_{1 \ll N\leq N_1 \sim N_2 } |I(f^{N}, g_1^{N_1}, g_2^{N_2})| \ls
  \|v\|_{X^{W\pm}_{-1-\sigma,\frac12,\infty}}
 \|u_1\|_{X^S_{s,\frac12,\infty}} \|u_2\|_{X^S_{s,b,\infty}}.
 \]

 {\it Case b)} very small wave frequency, i.e. $N\ls 1$. In this case,
 either $N_1\sim N_2$ or $N,N_1,N_2\ls 1$. We use \eqref{eq:very_small_wave_freq} and obtain
 \begin{align*}
   & \sum_{L,L_1,L_2\geq 1} |I( f^{L,N}, g_1^{L_1,N_1},g_2^{L_2,N_2})| \\
   \ls {} & \sum_{1 \leq L,L_1,L_2} (\frac{\min\{L,L_1,L_2\}}{L
   L_1 L_2})^{\frac12}
   \|P_N v \|_{X^{W\pm}_{0,\frac12,\infty}} \|P_{N_1}u_1\|_{X^S_{0,\frac12,\infty}}  \|P_{N_2}u_2\|_{X^S_{0,\frac12,\infty}} \\
   \ls {} &\|P_N v \|_{X^{W\pm}_{0,\frac12,\infty}} \|P_{N_1}u_1\|_{X^S_{0,\frac12,\infty}}  \|P_{N_2}u_2\|_{X^S_{0,\frac12,\infty}}.
 \end{align*}
 A similar argument shows
 \begin{align*}
  \sum_{L,L_1,L_2\geq 1} |I( f^{L,N}, g_1^{L_1,N_1},g_2^{L_2,N_2})| 
\ls{} & \|P_N v\|_{X^{W\pm}_{0,\frac12,1}}
   \|P_{N_1}u_1\|_{X^S_{0,\frac12,\infty}}\|P_{N_2}u_2\|_{X^S_{0,b,\infty}}.
 \end{align*}
provided that $b > 0$. \eqref{eq:trilinear1} and \eqref{eq:trilinear2} and their counterpart in ii)
 follow from these estimates since $N_1 \sim N_2$ or $N_1, N_2 \lesssim 1$.

 {\it Case c)} high-low interactions, i.e. $N_1\ll N_2$ or $N_1\gg
 N_2$. We focus on the case $N_1\ll N_2$, the other one being similar.
 Since we apply Proposition \ref{prop:high-low} we need to
 differentiate between the cases $L_2 \ll N_2^2$ and $L_2 \gs
 N_2^2$. In the first case, by \eqref{eq:high-low} and the observation \eqref{eq:res_in} which implies
 that $\max{(L,L_1)} \gs N_2^2$ for non-vanishing interactions, we have
 \begin{align*}
   & \sum_{L_2 \ll N_2^2} \sum_{L,L_1\geq 1} |I( f^{L,N}, g_1^{L_1,N_1},g_2^{L_2,N_2})| \\
   \ls {} & N_1 N_2^{-\frac32} \sum_{L_2 \ll N_2^2} \sum_{L,L_1\geq 1}
   \langle \frac{\max{(L,L_1)}}{N_2^2} \rangle^{-\frac12}
   \|P_N v\|_{X^{W\pm}_{0,\frac12,\infty}} \|P_{N_1}u_1\|_{X^S_{0,\frac12,\infty}}  \|P_{N_2}u_2\|_{X^S_{0,\frac12,\infty}} \\
   \ls {} & N_1 N_2^{-\frac32} (\ln{N_2})^2 \|P_N v\|_{X^{W\pm}_{0,\frac12,\infty}} \|P_{N_1}u_1\|_{X^S_{0,\frac12,\infty}}  \|P_{N_2}u_2\|_{X^S_{0,\frac12,\infty}}.
 \end{align*}
 By the same reasoning we also have
\begin{align*}
   & \sum_{L_2 \ll N_2^2} \sum_{L,L_1\geq 1} |I( f^{L,N}, g_1^{L_1,N_1},g_2^{L_2,N_2})| \\
   \ls {} & N_1 N_2^{-\frac12-2b} (\ln{N_2})^2 \|P_N v\|_{X^{W\pm}_{0,\frac12,\infty}} \|P_{N_1}u_1\|_{X^S_{0,\frac12,\infty}}  \|P_{N_2}u_2\|_{X^S_{0,b,\infty}}.
 \end{align*}

 In the case $L_2 \gs N_2^2$ we use \eqref{eq:high-low2} and obtain
 \begin{align*}
   & \sum_{L_2 \gs N_2^2} \sum_{L,L_1 \geq 1} |I( f^{L,N}, g_1^{L_1,N_1},g_2^{L_2,N_2})| \\
   \ls {}&{} N_1^{\frac12} \sum_{L_2 \gs N_2^2}  \sum_{1\leq L,L_1 \leq N_1^2}L_2^{-\frac12} \|P_N v\|_{X^{W\pm}_{0,\frac12,\infty}} \|P_{N_1}u_1\|_{X^S_{0,\frac12,\infty}}  \|P_{N_2}u_2\|_{X^S_{0,\frac12,\infty}}\\
& +N_1^{\frac32} \sum_{L_2 \gs N_2^2}  \sum_{\max\{L,L_1\}> N_1^2} (\max\{L,L_1\} L_2)^{-\frac12} \|P_N v\|_{X^{W\pm}_{0,\frac12,\infty}} \|P_{N_1}u_1\|_{X^S_{0,\frac12,\infty}}  \|P_{N_2}u_2\|_{X^S_{0,\frac12,\infty}}\\
\ls {}&N_1^{\frac12}(\ln N_1)^2  N_2^{-1} \|P_N v\|_{X^{W\pm}_{0,\frac12,\infty}} \|P_{N_1}u_1\|_{X^S_{0,\frac12,\infty}}  \|P_{N_2}u_2\|_{X^S_{0,\frac12,\infty}}
 \end{align*}
 In a similar manner we obtain
 \begin{align*}
   & \sum_{L_2 \gs N_2^2} \sum_{L,L_1 \geq 1} |I( f^{L,N}, g_1^{L_1,N_1},g_2^{L_2,N_2})| \\
   \ls & N_1^{\frac32-2b} (\ln N_1)^2 N_2^{-1}
   \|P_N v\|_{X^{W\pm}_{0,\frac12,\infty}} \|P_{N_1}u_1\|_{X^S_{0,\frac12,\infty}}  \|P_{N_2}u_2\|_{X^S_{0,\frac12,\infty}}
 \end{align*}

 For proving \eqref{eq:trilinear1} we estimate the above term in the
 worst case in which we place the low Schr\"odinger frequency in the
 space with positive Sobolev regularity and the high Schr\"odinger
 frequency in the space with negative Sobolev regularity. It is
 obvious that the other case gives better estimates. From the above
 inequalities we deduce
 \begin{align*}
   & \sum_{ N_1 \ll N \sim N_2} |I(f^{N}, g_1^{N_1}, g_2^{N_2})| \\
    \ls{}& \sum_{ N_1 \ll  N \sim N_2} N_1^{-s +\frac12} N^{s -1-\sigma}
   (\ln{N_1})^2  \|P_N v\|_{X^{W\pm}_{\sigma,\frac12,\infty}}
\|P_{N_1}u_1\|_{X^S_{s,\frac12,\infty}}
\|P_{N_2}u_2\|_{X^S_{-s,\frac12,\infty}}
 \end{align*}
 If $s \leq \frac12$, then we can bound the above sum by
 \begin{align*}
   & \|P_{N_1}u_1\|_{X^S_{s,\frac12,\infty}} \sum_{ N \sim N_2}
   N^{-\frac12-\sigma} (\ln{N})^3 \|P_N v\|_{X^{W\pm}_{\sigma,\frac12,\infty}}
\|P_{N_2}u_2\|_{X^S_{-s,\frac12,\infty}} \\
   \ls {} & \|P_N v\|_{X^{W\pm}_{\sigma,\frac12,\infty}}
\|P_{N_1}u_1\|_{X^S_{s,\frac12,\infty}}
\|P_{N_2}u_2\|_{X^S_{-s,\frac12,\infty}}
 \end{align*}
 where we have used that $\sigma > -\frac12$.
  
 If $s > \frac12$, then we can bound the above sum by
 \begin{align*}
   & \|P_{N_1}u_1\|_{X^S_{s,\frac12,\infty}} \sum_{ N \sim N_2} N^{s-1-\sigma}   \|P_N v\|_{X^{W\pm}_{\sigma,\frac12,\infty}}
\|P_{N_2}u_2\|_{X^S_{-s,\frac12,\infty}} \\
   \ls {} & \|P_N v\|_{X^{W\pm}_{\sigma,\frac12,\infty}}
\|P_{N_1}u_1\|_{X^S_{s,\frac12,\infty}}
\|P_{N_2}u_2\|_{X^S_{-s,\frac12,\infty}}
 \end{align*}
 where we have used that $s \leq 1+\sigma$. As noted earlier, it is
 obvious that in the case $ N_2\ll N_1\sim N $ the above estimate is easier.
With similar arguments we can verify the counterpart in ii) of these estimates, 
but we omit the details.
  
 Concerning \eqref{eq:trilinear2} we proceed as follows:
 \begin{align*}
   & \sum_{ N_1 \ll N \sim N_2} |I(f^{N}, g_1^{N_1}, g_2^{N_2})| \\
   & \ls \sum_{ N_1 \ll  N \sim N_2} N_1^{-s+\frac12} N^{-s+\sigma}
   (\ln N_1)^2  \|P_N v\|_{X^{W\pm}_{-1-\sigma,\frac12,\infty}}
\|P_{N_1}u_1\|_{X^S_{s,\frac12,\infty}}
\|P_{N_2}u_2\|_{X^S_{s,\frac12,\infty}}
 \end{align*}
 If $s \leq \frac12$, then we can bound the above sum by
 \begin{align*}
   & \|P_{N_1}u_1\|_{X^S_{s,\frac12,\infty}} \sum_{ N \sim N_2} N^{-2s+\frac12+\sigma} (\ln{N})^2 \|P_N v\|_{X^{W\pm}_{-1-\sigma,\frac12,\infty}}
\|P_{N_2}u_2\|_{X^S_{s,\frac12,\infty}}\\
   \ls {} &\|P_N v\|_{X^{W\pm}_{-1-\sigma,\frac12,\infty}} \|P_{N_1}u_1\|_{X^S_{s,\frac12,\infty}} 
\|P_{N_2}u_2\|_{X^S_{s,\frac12,\infty}}
 \end{align*}
 where we have used that $\sigma+\frac12 < 2s$.
  
 If $s > \frac12$, then we can bound the above sum by
 \begin{align*}
   & \|P_{N_1}u_1\|_{X^S_{s,\frac12,\infty}} \sum_{ N \sim N_2} N^{-s+\sigma}  \|P_N v\|_{X^{W\pm}_{-1-\sigma,\frac12,\infty}}
\|P_{N_2}u_2\|_{X^S_{s,\frac12,\infty}}\\
   \ls & \|P_N v\|_{X^{W\pm}_{-1-\sigma,\frac12,\infty}} \|P_{N_1}u_1\|_{X^S_{s,\frac12,\infty}} 
\|P_{N_2}u_2\|_{X^S_{s,\frac12,\infty}}
 \end{align*}
 where we have used that $\sigma \leq s$.
  
 It is an easy exercise to verify the counterpart in ii) of these estimates. 
 This concludes the proof of Proposition \ref{prop:trilinear}.
\end{proof}

\subsection*{Acknowledgments}
The first author was partially supported by the NSF grant DMS-1001676.
The first author would like to thank Todor Milanov for helpful discussions
related to the geometry of the problem. Moreover, the second author
would like to thank Jonathan Bennett for helpful discussions on
\cite{bennett_bez_2010}.
We are grateful to Justin Holmer for drawing our attention to the
Cauchy problem associated with the 3D Zakharov system.

\bibliographystyle{amsplain} \bibliography{zs-refs}

\end{document}